\documentclass[reqno]{amsart}

\usepackage{amssymb}
\usepackage[all]{xy}
\usepackage{caption}
\usepackage[pdftex]{hyperref}
\usepackage{graphicx}
\usepackage{array}
\usepackage{physics}
\usepackage{color}
\usepackage{stmaryrd}
\usepackage{xypic}
\usepackage{bm}
\usepackage{comment}

\newtheorem{thm}{Theorem}[section]
\newtheorem{lem}[thm]{Lemma}
\newtheorem{prop}[thm]{Proposition}
\newtheorem{cor}[thm]{Corollary}
\theoremstyle{definition}

\newtheorem{ex}[thm]{Example}
\newtheorem{rem}[thm]{Remark}

\newcommand{\sldos}{\SL_2(\mathbb{C})}
\newcommand{\sltres}{\SL_3(\mathbb{C})}
\newcommand{\CC}{{\mathbb C}}
\newcommand{\ZZ}{{\mathbb Z}}

\newcommand{\G}{{\Gamma}}
\newcommand{\R}{{\mathcal{R}}}
\newcommand{\X}{{\mathcal{X}}}
\newcommand{\cU}{{\mathcal{U}}}
\newcommand{\cV}{{\mathcal{V}}}
\newcommand{\cW}{{\mathcal{W}}}

\newcommand{\Free}[1]{\mathbf{F}_#1}

\DeclareMathOperator{\Id}{Id}
\DeclareMathOperator{\Hom}{Hom\,}           %Hom%
\DeclareMathOperator{\Stab}{Stab}

\DeclareMathOperator{\GL}{GL}
\DeclareMathOperator{\SU}{SU}

\DeclareMathOperator{\SL}{SL}

\DeclareMathOperator{\PGL}{PGL}

\DeclareMathOperator{\Gr}{Gr}
\DeclareMathOperator{\Sym}{Sym}    
\DeclareMathOperator{\diag}{diag}

\title{Character varieties of torus links}


\author{\'Angel Gonz\'alez-Prieto}
\address{Departamento de \'Algebra, Geometr\'ia y Topolog\'ia, Facultad de Ciencias Matem\'aticas, Universidad Complutense de Madrid, Plaza Ciencias 3, 28040 Madrid Spain, and Instituto de Ciencias Matem\'aticas (CSIC-UAM-UCM-UC3M), C.\ Nicolás Cabrera 13-15, 28049 Madrid Spain.}\email{angelgonzalezprieto@ucm.es}

\author{Javier Mart\'inez}

\address{Departamento de Matem\'atica Aplicada, Ciencia e Ingeniería de los Materiales y Tecnología Electrónica. E.S. Ciencias Experimentales y Tecnología, Universidad Rey Juan Carlos, C.\ Tulipán 0, 28933 Móstoles, Madrid Spain.}
\email{javier.martinezmar@urjc.es}

\author{Vicente Mu\~noz}
\address{Instituto de Matem\'aticas Interdisciplinar and 
Departamento de \'Algebra, Geometr\'ia y Topolog\'ia, Facultad de Ciencias Matem\'aticas, Universidad Complutense de Madrid, Plaza Ciencias 3, 28040 Madrid Spain.}\email{vicente.munoz@ucm.es}

\begin{document}

\begin{abstract}
In this paper, we study the geometry of the moduli space of representations of the fundamental group of the complement of a torus link into an algebraic group $G$, an algebraic variety known as the $G$-character variety of the torus link. These torus links are a family of links in the $3$-dimensional sphere formed by stacking several copies of torus knots. We develop an intrinsic stratification of the variety that allows us to relate its geometry with the one of the underlying torus knot. Using this information, we explicitly compute the $E$-polynomial associated to the Hodge structure of these varieties for $G = \textrm{SL}_2(\mathbb{C})$ and $\textrm{SL}_3(\mathbb{C})$, for an arbitrary torus link, showing an unexpected relation with the number of strands of the link.
\end{abstract}

\maketitle

%%%%%%%%%%%%%%%%%%%%%%%
\section{Introduction}\label{sec:introduction}
%%%%%%%%%%%%%%%%%%%%%%%
Fix a complex reductive algebraic group $G$. For any manifold $M$ with the homotopy type of a finite CW-complex, we can consider the moduli space $\X_G(M)$ parametrizing isomorphism classes of representations
$$
	\rho: \pi_1(M) \longrightarrow G
$$
of the fundamental group of $M$ into $G$. This moduli space has a natural complex algebraic structure and it is known in the literature as the $G$-character variety of $M$. 

These character varieties play an important role in algebraic geometry and theoretical physics. For instance, when $M = \Sigma$ is a closed surface, the celebrated non-abelian Hodge correspondence shows that the character variety $\X_G(\Sigma)$ is diffeomorphic to important moduli spaces attached to $\Sigma$, such as moduli spaces of flat connections on $\Sigma$ or moduli spaces of Higgs bundles, see \cite{Corlette:1988,NS,Simpson:1992,SimpsonI,SimpsonII}. These relations have been exploited to study some topological invariants of $\X_G(\Sigma)$, like its Betti numbers through a perfect Morse function naturally defined on the moduli space of Higgs bundles \cite{garciaprada-gothen,garciaprada-logares-munoz,hitchin}, or invariants of their complex structure using arithmetic techniques   \cite{Hausel-Rodriguez-Villegas:2008,mereb}, geometric techniques \cite{LMN,Martinez:2017,Martinez:2016} or topological recursion-based techniques \cite{GP-2019,GP-2018a,GPLM-2017}. 

Another fascinating horizon arises when we consider character varieties associated to a $3$-dimensional manifold $M$. In this case, $\X_G(M)$ gives information of the geometry of $M$, and can be used for instance to obtain results 
about metrics of constant curvature. For example, in the seminar work of Culler and Shalen \cite{Culler-Shalen} the authors used some simple algebro-geometric properties of the $\textrm{SL}_2(\CC)$-character variety to provide new proofs of Thurston’s theorem, that states that the space of hyperbolic structures on an acylindrical $3$-manifold is compact, or of the Smith Conjecture. 

A further interesting role of these varieties is as invariants of the $3$-manifold itself. For instance, we can take a link $L \subset \mathbb{S}^3$ and consider the character variety associated to the link complement $\X_G(L) := \X_G(\mathbb{S}^3 - L)$. This variety parametrizes the representations of the fundamental group $\pi_1(\mathbb{S}^3-L)$ of the link complement, and is thus invariant under isotopies of the link $L$. In particular, any invariant computed out of $\X_G(L)$ gives rise to a new link invariant for $L$. In this direction, Ekholm, Ng and Shende have proven that the Legendrian type of the conormal torus of a knot is a complete knot invariant \cite{Ekholm:2018}.

However, despite its importance in the geometry of $3$-manifolds, the lack of a known non-abelian counterpart in this context prevents the effective use of sophisticated techniques, such as Morse-theoretic tools, to understand character varieties of $3$-dimensional manifolds. For this reason, motivic invariants of $\X_G(M)$ can be computed, being the most important one the so-called $E$-polynomial
$$
	E(\X_G(M)) = \sum_{k, p, q} h_c^{k, p, q}(\X_G(M))\, u^pv^q \in \ZZ[u,v],
$$
where $h_c^{k, p, q}(\X_G(M)) = \dim_\CC H_c^{k, p, q}(\X_G(M))$ are the compactly-supported Hodge numbers of the Hodge structure of the cohomology of $\X_G(M)$.

These $E$-polynomials exhibit very useful properties of additivity with respect to stratifications and multiplicativity with respect 
to trivial monodromy fibrations that make their computation feasible using geometric techniques, although they are typically hard and require a specific-case analysis. Some known cases of these $E$-polynomials of character varieties are for torus knots and $G= \GL_2(\CC)$ \cite{LiXu}, $\SL_2(\CC)$ \cite{mun:2009}, $G = \SL_3(\CC)$ \cite{munozporti:2016} or $G = \SL_4(\CC)$ \cite{gonzalez2020motive}; and for the figure-eight knot and $G = \GL_3(\CC), \SL_3(\CC)$ and $\PGL_3(\CC)$ \cite{HMP}. Some works also exist in the literature for real groups, such as for torus knots and $G = \SU(2)$ \cite{martinezmun2015su2} or $G = \SU(3)$ \cite{GPMM}. Special mention deserves the case of trivial links, which is the only case fully understood for general rank for $G= \GL_n(\CC)$ or $\PGL_n(\CC)$ \cite{mozgovoyreineke:2015} and $\SL_n(\CC)$ \cite{florentinonozadzamora:2021}, see also \cite{florentino2009topology,Florentino-Lawton:2012}. In sharp contrast, almost nothing is known in the case of non-trivial links with several components, due to the much more intricate geometry of their character variety, being the only known examples the so-called twisted Hopf links, which are links with two components, for $G = \SL_2(\CC)$ and $\SL_3(\CC)$ \cite{gonmun:2022} and $G = \SU(2)$ and $\SU(3)$ \cite{gonlogmarmun:2023}. In this setting, related to this paper, \cite{lawtonflorentino:2024} analyzes the number of irreducible components and path-connectedness of the $\SL_2(\CC)$-character variety of the so-called generalized torus knot groups, which includes the torus knot and link cases. 

The aim of this paper is to study the character varieties of a non-trivial family of links with arbitrary many components: the so-called torus links $K_{n,m}^d$ for $n, m, d \geq 1$, with $n$ and $m$ coprime. These are links with $d$ components given by placing $d$ parallel copies of an $(n,m)$-torus knot, i.e., the knot generated in the torus by the straight line of slope $n/m$ in its universal cover. These torus links generalize the usual torus knots for $d = 1$ and twisted Hopf links for $d = 2$ and $m=1$. To study these representation varieties, we shall exploit the fact that there exists a map
$$
	\pi: \X_G(K_{n,m}^d) \to \X_G(K_{n,m}^1)
$$
given by restriction to the first strand of the torus link. This map $\pi$ is far from being locally trivial, since its fibers depend on the $G$-stabilizer of the representation of the base torus knot.

However, in this work we develop a suitable stratification by stabilizers of these character varieties that allows to understand the geometry of the map $\pi$ stratum-wise for arbitrary group $G$. In this way, the knowledge of the geometry of character varieties of torus knots can be wrapped to provide $E$-polynomials for torus links. We apply these techniques in the case of $G = \SL_2(\CC)$ and $\SL_3(\CC)$ to get explicit expressions of the $E$-polynomials of these character varieties for arbitrary $d$, and coprime $n$ and $m$, obtaining the main result of this paper. To the best of our knowledge, this is the first work computing these motivic invariants for a non-trivial family of links. In the following, notice that $\X_G(K_{n,m}^d) \cong \X_G(K_{m,n}^d)$ so, without loss of generality, we can suppose that $m$ is odd.

\begin{thm}
Let $\X_G(K_{n,m}^d)$ be the $G$-character variety of the torus link $K_{n,m}^d$ with $d, n, m \geq 1$, $n$ and $m$ coprime, and $m$ odd. Then, for $G = \SL_2(\CC)$ we have that the $E$-polynomial of the character variety is
\begin{align*}
e(\X_{\SL_2(\CC)}(K_{n,m}^d))  & =  \frac{1}{2}(m-1)(n-1)(q-2)(q^3-q)^{d-1} \\
& \quad +(mn-1)\left((q^2+q)(q^3-q)^{d-2}-(q-1)^{d-2}(2q^{d-1}-1) \right) \\
& \quad + 2\left( (q^3-q)^{d-2}-(q^2-q)^{d-2}+ \frac{1}{2}q((q+1)^{d-2}+(q-1)^{d-2}) \right) \\
& \quad + (q-2)\frac{1}{2}\left( (q+1)^{d-1}+(q-1)^{d-1} \right) +\frac{1}{2}\left( (q+1)^{d-1}-(q-1)^{d-1} \right).
\end{align*}
Additionally, for $G = \SL_3(\CC)$ we have

\begin{align*}
e(\X_{\SL_3(\CC)}(K_{n,m}^d)) = & \phantom{+}  \frac{(q^3-q)^{d-1}}{12}(q^5-q^3)^{d-1} \left((m-1)(m-2)(n-1)(n-2) (q^4+4q^3-9q^2-3q+12) \right.  \\  & \left.  + \,6(n-1)(m-1)(n+m-4)(q^2-3q+3) \right) \\
& + 3mn\left(\left\lfloor\frac{m-1}{2}\right\rfloor \left\lfloor \frac{n-1}{2}\right\rfloor(q-2)+\delta_n(m-1)(q-1)\right) (q^3-q)^{d-1} \\ & \quad \left(q^{3d-3}(q+1)^{d-1}(q-1)^{d-2} -(2q^{2d-2}-1)(q-1)^{d-2}+(q-1)^{d-1} \right) \\ 
& + (q-3mn-1)\left(\left\lfloor\frac{m-1}{2}\right\rfloor\left\lfloor\frac{n-1}{2}\right\rfloor(q-2)+\delta_n(m-1)(q-1)\right) (q-1)^{2d-2}(q^2+q)^{d-1} \\ 
& +3(q^8-q^6-q^5+q^3)^{d-2} + 3(q-1)^{2d-4}(q^{3d-6}-q^{d-1}) \\ & +\frac{1}{2}(q-1)^{2d-4}q(q+1)+\frac{3}{2}(q^2-1)^{d-2}q(q-1) \\ &  
+(q^2+q+1)^{d-2}q(q+1)-3(q-1)^{d-2}q^{d-2}(q^2-1)^{d-2}(2q^{2d-4}-q)  \\ 
& + (3mn-3)\left( \vphantom{\frac{1}{3}} \frac{1}{2}\left( (q-1)^{2d-2} + (q^2-1)^{d-1}\right) + (q-1)^{2d-3}((q^2+q)^{d-1} - 2q^{d-1} +1) \right. \\& + (q - 1)^{d-1}\big((q - 1)^{d-2}q^{d-2}((q + 1)^{d-2} - 1) + \frac12 (q - 1)^{d-2} - \frac12 (q + 1)^{d-2} \big)   \\ &   + (q-1)^{2d-4} \left( (q^2+q+1)^{d-1}(q+1)^{d-2}q^{3d-4}-(q^2+q)^{d-2}(2q^{2d-2}-1) \right. \\ &  \left. -\, q^{d-2}(q^{d-1} - 1)^2 - (2q^{2d-3}-1) ((q^2+q)^{d-1} - 2q^{d-1} + 1))  \vphantom{\frac{1}{3}}\right) \\
& +(q-3mn-1)(q-1)^{d-1}\left( (q^3-q)^{d-2}-(q^2-q)^{d-2}+ \frac{1}{2}q(q+1)^{d-2}+\frac{1}{2}q(q-1)^{d-2} \right) \\ & + \frac{m^2n^2-3mn+2}{2} \left( \vphantom{\frac{1}{3}} (q-1)^{2d-2}+3(q-1)^{2d-3}((q^2+q)^{d-1}-2q^{d-1}+1) \right. \\ 
	& \left. +\,(q-1)^{2d-4}\left((q^2+q+1)^{d-1}(q+1)^{d-1}q^{3d-3}-3(q+1)^{d-1}(2q^{3d-3}-q^{d-1}) \right. \right. \\ & \left. \left. +6q^{d-1}(q^{2d-2}-1)+2  \right) \vphantom{\frac{1}{3}} \right) + (mn-1)\left( q- 3mn-1 \right) \left(\frac{1}{2} q{\left(q - 1\right)}^{2d-3} - q^{d-1} {\left(q - 1\right)}^{2d-3}\right) \\ & + \frac{\delta_{nm}}{2} q(q-1)^{d} (q + 1)^{d-2}  + \left( q\left\lfloor \frac{mn}{2} \right\rfloor + \left\lfloor \frac{mn-1}{2} \right\rfloor\right) {q^{d-1}\left(q + 1\right)}^{d-2} {\left(q - 1\right)}^{2d-2} 
 \\  & -\frac{3mn(mn-1)}{2} {q^{d-1}\left(q + 1\right)}^{d-1} {\left(q - 1\right)}^{2d-3} + \frac{(q^2-1)^{d-1}}{2}(q-1)\left(q-\delta_{mn} \right)  \\ & + \frac{1}{6}(q-1)^{2d-2} \left((q-1)\left(q-3mn-1\right) +6m^2n^2 \right) +\frac{1}{3}(q^2+q+1)^{d-1} (q-1)(q+2),
\end{align*}
where $\delta_k=1$ if $k$ is even and $\delta_k=0$ if $k$ is odd.
\end{thm}

As we will see, the calculation of these $E$-polynomials motivates the study of character varieties for which the isomorphism relation must preserve a given flag, a new kind of character varieties that we baptize as $\lambda$-character varieties. These varieties seem to play an important role in the geometry of character varieties of links, so we encourage further research on their own.

An interesting and unexpected feature of the result obtained is that $d$ and $n,m$ play very different roles in the final expression of the $E$-polynomial. Whereas $n$ and $m$ always appear as coefficients of this polynomial, $d$ is presented in the exponent of some of these polynomials. This is a very suggestive behavior that seems to point out that a kind of Topological Field Theory may exist to compute these $E$-polynomials of torus links with $d$ strands out of the knowledge of torus knots, in the same vein that the appearance of the genus in the formula of the $E$-polynomial of character varieties of surfaces pointed to the construction of a Topological Quantum Field Theory for them \cite{GPLM-2017,Martinez:2016}.

Finally, we want to stress that the techniques developed in this paper can be extended in several directions. First, the stratification constructed for the character variety of torus links works for general groups $G$, in particular for higher rank groups $\SL_n(\CC)$. However, the number of strata arising in this decomposition grows exponentially with the rank, so a computer-based approach would be necessary to deal with the higher rank case. Additionally, the geometry of $\SL_n(\CC)$-character varieties of torus knots, which are needed to pass to torus links, is only known for rank $n \leq 4$ so far. Hence, a more intensive research on the geometry of high rank character varieties of torus knots is needed. But, more importantly, the tools we have developed in this work can be applied to other families of links, such as links of parallelly placed knots, in which no further difficulties are expected and the techniques should be applicable verbatim. 

\subsection*{Structure of the manuscript} The structure of this manuscript is as follows. In Section \ref{sec:introduction} we briefly review the fundamentals of character varieties and their Hodge structures on cohomology, including the relevant properties of $E$-polynomials. In Section \ref{sec:representationstrata}, we zoom-in to character varieties of torus links, developing a stratification technique to relate their geometry with the one of the underlying torus knot. With these tools, in Section \ref{sec:sl2geodesc} we address the case of $\SL_2(\CC)$-character varieties, computing their $E$-polynomials. As an intermediate step towards the higher rank setting, in Section \ref{sec:lambda-character-varieties} we develop the theory of $\lambda$-character varieties and compute explicitly some examples. The stratification for the $\textrm{SL}_3(\CC)$-character varieties are described explicitly in Section \ref{sec:sl3charactervariety}. To compute the $E$-polynomial of each stratum, several auxiliary calculations are carried out in Section \ref{sec:git-quotients}, and the final computation is shown in Section \ref{sec:E-poly-rank3}.

\subsection*{Acknowledgements}

The first author acknowledges Alejandro Calleja for very useful discussion around character varieties of torus knots. The first author has been partially supported by Comunidad de Madrid R+D Project PID/27-29 and Ministerio de Ciencia e Innovaci\'on Project PID2021-124440NB-I00 (Spain), the second author has been partially supported by Comunidad de Madrid multiannual agreement with Universidad Rey Juan Carlos under the grant Proyectos I+D para Jóvenes Doctores, ref. M2731, project NETA-MM, and the third author has been partially supported by Comunidad de Madrid R+D Project PID/27-29 and Ministerio de Ciencia e Innovaci\'on Project PID2020-118452GB-I00 (Spain). Several calculations were assisted with the programming language SageMath.

%%%%%%%%%%%%%%%%%%%%%%%
\section{Representation varieties and their Hodge structures}\label{sec:introduction}
%%%%%%%%%%%%%%%%%%%%%%%

Given a finitely presented group $\Gamma$ and a reductive complex algebraic group $G$, a \emph{representation} of $\Gamma$ into $G$  is a group homomorphism $\rho: \Gamma \to G$. If $\Gamma=\langle x_1,\ldots, x_k \mid r_1,\ldots,r_s \rangle$ is a presentation for $\Gamma$, the space of all representations can be explicitly described as
$$
\R_G(\Gamma)= \Hom(\Gamma,G)= \lbrace (A_1,\ldots, A_k) \mid r_j(A_1,\ldots,A_k)=\Id, 1\leq j \leq s \rbrace \subset G^k,
$$
which is an affine complex algebraic set. Two representations $\rho_1,\rho_2$ are declared equivalent if there exists $g\in G$ such that $\rho_1(\gamma)=g\rho_2(\gamma)g^{-1}$ for all $\gamma \in \Gamma$. If we regard $G\subset \GL_r(\CC)$ as a linear group, this equivalence corresponds to a $G$-change of basis in $\CC^r$ of the representation. The moduli space of representations, the \emph{character variety}, is the GIT quotient with respect to the conjugacy action
$$
\X_G(\Gamma)=\R_G(\Gamma) \sslash G.
$$

We consider throughout the paper the following finitely presented groups:
\begin{align}
\Gamma^d(n,m) & = \langle a,b,f_1,\ldots, f_{d-1} \mid a^n f_k = f_k b^m, k=0,\ldots ,d-1 \rangle, \label{eqn:toruslinkgroup} 
\end{align}
where $n,m$ and $d$ are positive integers and it is implicitly assumed that $f_0=1$ in \eqref{eqn:toruslinkgroup}.

To shorten notation, we shall also denote $\Gamma(n,m) = \Gamma^1(n,m) = \langle a, b \,|\, a^n = b^m \rangle$. When $n$ and $m$ are coprime, this group represents the \emph{knot group} of a torus knot, that is, the fundamental group of the knot complement $\Gamma(n,m)=\pi_1(\mathbb{S}^3-K_{n,m})$, where $K_{n,m}$ is the simple closed curve on the torus that is the image of the line of rational slope $n/m$ in the square representation of the torus, which we assume embedded in $\mathbb{S}^3$.

In general, we can consider $d$ copies of the knot $K_{n,m}$ placed parallel to each other, leading to a link with $d$ components denoted by $K_{n,m}^d$, as depicted in Figure \ref{fig:turuslink}. This torus link is sometimes also denoted as $K_{dm,dn}$, extending the previous definition to the non-coprime case, where  the link with $d$ connected components is obtained by conveniently identifying certain strands on the cylinder when its opposite ends are glued to obtain a torus \cite{mur:2008}. It is shown in \cite{argyreskulkarni:2019} that the link group of $\mathbb{S}^3-K^d_{n,m}$ is precisely $\pi_1(\mathbb{S}^3-K^d_{n,m}) = \Gamma^d(n,m)$.

\begin{figure}[h]\label{fig:turuslink}
\includegraphics[width=9cm]{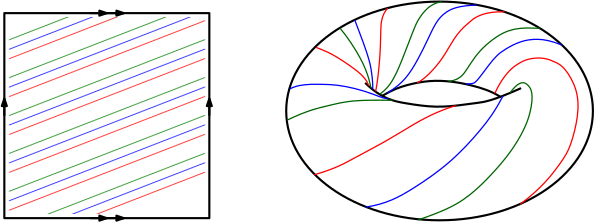}
\caption{Torus link $K^3_{2,5}$. It is composed of three strands that wrap the torus twice around the parallels and five times around the meridians.}\label{fig:turuslink}
\end{figure}

\begin{rem} 
For $n \geq 1$, the twisted Hopf link group,
$$
H_n=\langle a,b \mid [a^n,b]=1 \rangle,
$$
studied in \cite{gonmun:2022} and \cite{gonlogmarmun:2023}, corresponds to $\Gamma^2(n,1)$, since
$$
\Gamma^2(n,1)  =\langle a,b,f_1 \mid a^n=b, \, a^n f_1 = f_1 b \rangle = \langle a, f_1 \mid a^n f_1 = f_1 a^n \rangle = \langle a, f_1 \mid [a^n,f_1]=1 \rangle.
$$
Notice that here $[a,b] = aba^{-1}b^{-1}$ denotes the group commutator.
\end{rem}

%%%%%%%%%%%%%%%%%%%%%%%%%%%%%%%%%%%%%%%%%%%
\subsection{$E$-polynomials}
%%%%%%%%%%%%%%%%%%%%%%%%%%%%%%%%%%%%%%%%%%%

Given any quasi-projective complex variety $Z$, its $E$-polynomial is defined as
$$
e(Z)=\sum_{p,q,k} (-1)^kh_c^{k,p,q}(Z)u^pv^q,
$$
where $h_c^{k,p,q}(Z)$ are the \emph{(mixed) Hodge numbers} of $Z$. Deligne proved \cite{deligne:1971} that the cohomology groups $H^k(Z)$ and $H^k_c(Z)$ admit so-called mixed Hodge structures, which are defined as an (ascending) weight filtration $\ldots \subset W_{l-1}\subset W_l \subset \ldots \subset H^k_c(Z)$ and a (descending) filtration $F^{\bullet}$ that induces a pure Hodge structure of weigth $s$ on each graded piece $\Gr_s^W(H^k_c(Z))=W_s/W_{s-1}$. Both filtrations provide the \emph{mixed Hodge numbers} of $Z$ as
$$
h_c^{k,p,q}(Z)=\dim \Gr_F^p \Gr_{p+q}^W H_c^k(Z),
$$
extending the classical Hodge numbers of complex compact K\"ahler varieties to the non-smooth or non-compact cases. For more details, see \cite{Peters-Steenbrink:2008}.

These $E$-polynomials have been extensively studied in recent years, since they encode topological, algebraic and arithmetic information of the underlying variety. In many cases $h_c^{k,p,q}(Z)=0$ if $p\neq q$, so that that the variable $q=uv$ can be used. Varieties with this property are said to be of \emph{balanced type}, and some examples are:
\begin{itemize}
\item $e(\CC^r)=q^r$,
\item $e(\GL_r(\CC))=(q^r-1)(q^r-r)\ldots(q^r-q^{r-1})$,
\item $e(\SL_r(\CC))=(q^r-1)(q^r-q)\ldots(q^r-q^{r-2})q^{r-1}$. 
\end{itemize}
In particular, $e(\sldos)=q^3-q$ and $e(\sltres)=(q^3-1)(q^3-q)q^2=(q^3-1)(q^5-q^3)$.

We list here some of the basic properties of $E$-polynomials that will be applied (for proofs, please refer to \cite{LMN}). 
%The first one is particular important for our purposes, since it will allow us to use the stratifications of $\X^d_{\sldos}(n,m)$ that were described in Section \ref{sec:sl2geodesc}:
\begin{itemize}
\item $e(Z)=\sum_{i=1}^n e(Z_i)$ if $Z=\sqcup_{i=1}^n Z_i$ and $Z_i$ are locally closed in $Z$.
\item $e(Z)=e(F)e(B)$ for any fiber bundle $F\rightarrow Z \rightarrow B$ in the analytic topology that induces trivial monodromy action on the cohomology of the fiber. Notice that this hypothesis holds trivially in several cases, such as when the fiber bundle is also locally trivial in the Zariski topology or for principal bundles where the gauge group is a connected algebraic group. 
\end{itemize}

This definition can be extended to the equivariant setting. When a finite group $F$ acts on a complex variety $Z$, it also acts on $H^{\ast}_c(Z)$ respecting its mixed Hodge structure. In particular, an equivariant version $e_F(Z)$ of the $E$-polynomial can be defined, 
$$
e_F(Z)=\sum_{p,q,k}(-1)^k[H^{k,p,q}_c(Z)]\,u^pv^q \in R(F)[u,v],
$$
where the coefficients $[H^{k,p,q}_c(Z)]$ belong now to the representation ring $R_\CC(F)$ of virtual representations of $F$. Hence, we can always write in a unique way
$$
	e_F(Z)=\sum a_j T_j,
$$
where $T_j$ are the irreducible characters in $R_\CC(F)$ with $T_0$ being the trivial representation. This is particularly useful to analyze $F$-quotients of $Z$, since we have that $H^*_c(Z/F) = H^*_c(Z)^F$, the $F$-fixed part of the cohomology, and thus $e(Z/F)$ is just the coefficient $a_0$ in the sum above. A straightforward calculation shows that the equivariant $E$-polynomial satisfies the same additive and multiplicative properties as the usual $E$-polynomial. For more information about the equivariant setting, see \cite[Section 4]{florentinosilva}.

\begin{ex}\label{ex:pm-formula}
Fix $F=\ZZ_2 = \ZZ/2\ZZ$ and let $T$ and $N$ be the trivial and non-trivial irreducible characters. Then, the equivariant $E$-polynomial of any variety $Z$ acted by $\ZZ_2$ can be written as $e_{\ZZ_2}(Z)=aT+bN$. Now, let $Z_1,Z_2$ be varieties with $\ZZ_2$-actions, so that $e_{\ZZ_2}(Z_i)=a_iT+b_iN$ for $i=1,2$. Notice that $T \otimes T = T$, $T \otimes N = N$ and $N \otimes N = T$, and thus 
$$
e_{\ZZ_2}(Z_1\times Z_2)= e_{\ZZ_2}(Z_1) \otimes e_{\ZZ_2}(Z_2)= (a_1a_2+b_1b_2)T+(a_1b_2+a_2b_1)N.
$$
In particular,
\begin{equation}\label{eqn:epolyZ2}
e((Z_1\times Z_2)/\ZZ_2)= e(Z_1\times Z_2)^{\ZZ_2} =a_1a_2+b_1b_2=e(Z_1)^+e(Z_2)^++e(Z_1)^-e(Z_2)^-,
\end{equation}   
where $e(Z)^+ = e(Z / \ZZ_2),e(Z)^- = e(Z) - e(Z)^+$ denote the $E$-polynomials of the invariant and non-invariant part of the cohomology, respectively. 
\end{ex}

\begin{rem}
Since equivariant $E$-polynomials are multiplicative for Zariski-locally trivial bundles, the same formula holds in the case of a $\ZZ_2$-equivariant fiber bundle. Indeed, if $F \to X \to B$ is such a bundle, then we have
\begin{equation}\label{eqn:pm-formula-fiber}
	e_{\ZZ_2}(X) = \left(e(B)^+e(F)^+ + e(B)^-e(F)^-\right)T + \left(e(B)^+e(F)^-+e(B)^-e(F)^+\right)N.
\end{equation}
\end{rem}

\begin{rem}\label{rem:solve-for-bundle}
Formula (\ref{eqn:pm-formula-fiber}) can be used reciprocally to compute the $\ZZ_2$-equivariant $E$-polynomial of $B$ given the ones of $F$ and $X$. Indeed, a direct calculation shows that
\begin{equation}\label{eqn:pm-formula-fiber-forbase}
	e_{\ZZ_2}(B) = \frac{e(X)^+e(F)^+ - e(X)^-e(F)^-}{(e(F)^+)^2 - (e(F)^-)^2}T + \frac{e(X)^-e(F)^+-e(X)^+e(F)^-}{(e(F)^+)^2 - (e(F)^-)^2}N.
\end{equation}
\end{rem}

\begin{ex}\label{ex:6}
Fix $F=S_3$, the symmetric group in three elements. 
Denote by $\alpha=(1\, 2 \, 3)$ the $3$-cycle and $\tau=(1\,2)$ a
transposition. There are three irreducible representations
$T,S,D$ of $S_3$, where $T$ is the trivial one, $S$ is the sign representation,  and $D$ is the standard representation.
The sign representation is one-dimensional $S=\CC$, where $\alpha\cdot x=x$ and $\tau\cdot x=-x$.
The standard representation is two-dimensional $D=\CC^2$, where
$$
	\alpha \mapsto \begin{pmatrix}0 & -1 \\ 1 & -1\end{pmatrix}, \qquad \tau \mapsto \begin{pmatrix} 0 & 1 \\ 1 & 0\end{pmatrix}.
$$
%$\tau\cdot z=\overline{z}$, $\alpha\cdot z=e^{2\pi i/3} z$.
The multiplicative table of $R_\CC(S_3)$ is easily checked to be given by 
 \begin{align*}
  T\otimes T &=T,  &&   T\otimes S =S, \\
  T\otimes D &=D, &&  S\otimes S =T, \\
  S\otimes D &=D, &&  D\otimes D =T +S + D.
 \end{align*}
\end{ex}

%%%%%%%%%%%%%%%%%%%%%%%%%%%%%%%%%%%%%%%%%%%
\section{Stratification of representation varieties of torus links} \label{sec:representationstrata}
%%%%%%%%%%%%%%%%%%%%%%%%%%%%%%%%%%%%%%%%%%%

Let us fix a reductive complex algebraic group $G$ and let us denote the associated representation and character varieties by
$$
	\R_G^d(n,m) = \Hom(\Gamma^d(m, n), G), \qquad \X_G^d(n,m) = \Hom(\Gamma^d(n,m), G) \sslash G.
$$
We shall drop the mention to $G$ in the notation unless necessary.
We abbreviate $\R(n,m) = \R^1(n,m)$ and $\X(n,m) = \X^1(n,m)$. Observe that the representation variety is given explicitly by
$$
	\R^d(n,m) = \lbrace(A,B,F_1,\ldots,F_{d-1})\in G^{d+1} \mid A^n=B^m, \: F_iA^nF_i^{-1}=A^n,\: i=1,\ldots,d-1 \rbrace.
$$

The inclusion map $i: \G(n,m)\longrightarrow \G^d(n,m)$ induces a $G$-equivariant surjective map
\begin{equation}\label{eq:proj-knot-link}
	\pi: \R^d(n,m) \to \R(n,m), \qquad \pi(A,B, F_i) = (A,B).
\end{equation}
The fiber of this map over $(A,B) \in \R(n,m),$ is $\Stab_{G}(A^n)^{d-1}$, the stabilizer of $A^n$ in $G$ under conjugation. This map is not locally trivial in the Zariski topology. Indeed, its fibers vary from point to point.

To analyze it, let us consider the Luna stratification $\{U_H\}$ of $\R(n,m)$ under the adjoint action. To be precise, we have that $U_H \subset \R(n,m)$ is the collection of representations $\rho \in \R(n,m)$ such that the $G$-stabilizer of $\rho$ under the conjugacy action is conjugated to the subgroup $H < G$. By the Luna slice theorem, for each $U_H$, we have a commutative diagram
$$ 
\xymatrix{
\widetilde{U}_H \ar[d] \ar[r] &U_H \ar[d] \\
\widetilde{\mathcal{U}}_H := \widetilde{U}_H \sslash G \ar[r] & \mathcal{U}_H := U_H \sslash G 
}
$$
where the morphisms $\widetilde{\mathcal{U}}_H \to \mathcal{U}_H$ and $\widetilde{U}_H \to U_H$ are \'etale maps, and the vertical arrows are the quotient maps.

We shall need an explicit description that can be given as follows. Given a subgroup $H < G$, let $N_G(H) = \{g \in G \,|\, gHg^{-1} = H\}$ be the normalizer of $H$. Clearly, $H$ is a normal subgroup of $N_G(H)$, so we can form the quotient $N_H = N_G(H) / H$.

\begin{prop}\label{prop:structure-character-var}
Consider the subvariety $U_H^0 = \{\rho \in U_H \,|\, \Stab_G(\rho) = H\}$. Then we have a commutative diagram
$$ 
\xymatrix{
\widetilde{U}_H =  U_H^0 \times G/H \ar[d] \ar[r] &U_H = \left(U_H^0 \times G/H \right) \sslash N_H \ar[d] \\
\widetilde{\mathcal{U}}_H = U_H^0 \ar[r] & \mathcal{U}_H  = U_H^0 \sslash N_H
}
$$
where the action of $N_H = N_G(H) / H$ on $U_H^0 \times G/H$ is given by $nH \cdot (\rho, gH) = (n \rho n^{-1}, gn^{-1}H)$.
%$$
%	\widetilde{U}_H = \left(U_H^0 \times G/H \right)  \to \widetilde{\mathcal{U}}_H = U_H^0 ,
%$$
%$$
%	{U}_H = \left(U_H^0 \times G/H \right) \sslash N_H \to {\mathcal{U}}_H = U_H^0 \sslash N_H,
%$$
%where $N_H=N_G(H)/C_G(H)$.
\end{prop}

\begin{proof}
%Moreover, let $N_H := N_G(H)/H$ be the Weyl group of $H$ in $G$, where $N_G(H)$ is the normalizer of $H$ in $G$. \color{red} Javier :diría que en la definición va en el cociente $C_G(H)$, el centralizador. En el caso en que $H$ es un toro maximal coincide con $H$, pero no en otros casos (ver ejemplo de $\sldos$). \color{black}. 
We have a natural morphism
$$
	U_H^0 \times G \to U_H, \qquad (\rho, g) \mapsto g \rho g^{-1}.
$$
Since $H$ is exactly the stabilizer of the elements of $U_H^0$, this map descends to a morphism $U_H^0 \times G/H \to U_H$. Furthermore, a direct calculation shows that two elements $(\rho, gH)$ and $(\rho', g'H)$ of $U_H^0 \times G/H$ have the same image under this map if and only if there exists $n \in N_G(H)$ such that $(n \rho n^{-1}, gn^{-1}H)= (\rho', g'H)$. Indeed, if $g \rho g^{-1} = g' \rho' (g')^{-1}$, this implies that $\rho' = n \rho n^{-1}$ for $n = (g')^{-1}g$ and thus, since the stabilizer of $\rho'$ is exactly $H$, this means that $h\rho'h^{-1} = h n \rho n^{-1} h^{-1} = n\rho n^{-1}$ for all $h \in H$ or, equivalently $n^{-1}h n$ stabilizes $\rho$ for all $h$. But the stabilizer of $\rho$ is also $H$, so $n^{-1}h n \in H$ for all $h \in H$, meaning that $n$ must lie in the normalizer $N_G(H)$.

The so-defined action $n \cdot (\rho, gH) = (n \rho n^{-1}, gn^{-1}H)$ of $n \in N_G(H)$ on $(\rho, gH) \in U_H^0 \times G/H$ is well-defined precisely since $n \in N_G(H)$ and is trivial exactly on $H$. Hence we have an induced action of $N_H = N_G(H)/H$.
% $H \triangleleft N_G(H)$ is trivial, the result follows. \color{red} Javier: Diría que hace falta aquí de nuevo $C_G(H)$, que no actúa trivialmente, pero mantiene en $U_H^0$. \color{black}
%Hence, if we define the action of $N_H$ 
%on $U_H^0 \times G/H$ by $n \cdot (\rho, gH) = (n \rho n^{-1}, ngn^{-1}H)$.
We get the map
\begin{equation}\label{eq:orbit-space-rep}
	\widetilde{U}_H := U_H^0 \times G/H \to U_H=\left(U_H^0 \times G/H \right) \sslash N_H.
\end{equation}

Taking the quotient by $G$, we get a map at the level of quotients
\begin{equation}\label{eq:orbit-space-char}
	\widetilde{\mathcal{U}}_H := (U_H^0 \times G/H )\sslash G  \to \mathcal{U}_H = U_H \sslash G= U_H^0 \sslash N_H,
\end{equation}
as the $G$-action is given by $g'\cdot (\rho, gH) =  ( \rho , g'gH )$.
\end{proof}

%\color{red}Es $N_H = N_G(H) / H$ y no $N_G(H)/C_G(H)$, verdad? \color{black}

\begin{ex}
Let $G = \GL_r( \CC)$ and $H = \CC^* \Id$, so that $N_H = \GL_r( \CC)/\CC^* = \PGL_r(\CC)$. Schur's lemma implies that the set of irreducible representations $\R(n,m)^* \subset \R(n,m)$ is contained in $U_H$, so the irreducible characters $\X(n,m)^* \subset \X(n,m)$ are also contained in $\mathcal{U}_H = U_H \sslash G$. Furthermore, in this case $U_H^0 = U_H$, hence the commutative diagram of Proposition \ref{prop:structure-character-var} boils down to
$$ 
\xymatrix{
\widetilde{U}_H =  U_H \times \PGL_r(\CC) \ar[d] \ar[r] &U_H \supset \R(n,m)^* \ar[d] \\
\widetilde{\mathcal{U}}_H = U_H \ar[r] & \mathcal{U}_H  = U_H \sslash \PGL_r( \CC) \supset \X(n,m)^*
}
$$
where the top and left quotient maps reduce to projections onto the first factor.
\end{ex}

\begin{rem}
The definition of the group $N_H = N_G(H) / H$ resembles the definition of the Weyl group of $G$ and, indeed, this is the case when $H = T$ is a maximal torus of $G$. In this sense, the previous result can be seen as a generalization of the fact that there is a commutative diagram
$$ 
\xymatrix{
T \times G/T \ar[d] \ar[r] &G \ar[d] \\
T \ar[r] & G \sslash G = T \sslash N_T. 
}
$$
\end{rem}
 
Now, let us come back to the torus link representation variety $\R^d(n,m)$. Given subgroups $H_1, H_2 < G$ such that $H_1 \triangleleft H_2$, let us consider the strata
$$
V_{H_1, H_2} = \left\{(A, B) \in \R(n,m)\,|\, (\Stab(A,B), \Stab A^n) \sim (H_1, H_2)\right\},
$$
$$
W_{H_1, H_2} = \left\{(A, B, F_1, \ldots, F_{d-1}) \in \R^d(n,m)\,|\, (\Stab(A,B), \Stab A^n) \sim (H_1, H_2)\right\}.
$$
Here, $(H_1, H_2) \sim (H'_1, H'_2)$ means that $H_1$ and $H'_1$, and $H_2$ and $H_2'$ are simultaneously conjugated subgroups. Additionally, in analogy with the previous case, we set
$$
V_{H_1, H_2}^0 = \left\{(A, B) \in \R(n,m)\,|\, \Stab(A,B) = H_1, \, \Stab A^n = H_2\right\},
$$
$$
W_{H_1, H_2}^0 = \left\{(A, B, F_1, \ldots, F_{d-1}) \in \R^d(n,m)\,|\, \Stab(A,B) = H_1, \, \Stab A^n = H_2\right\}.
$$
Observe that, by definition, we have $W_{H_1, H_2}^0 = V_{H_1, H_2}^0 \times H_2^{d-1}$. 

Now, the isomorphism analogous to (\ref{eq:orbit-space-rep}) becomes something slightly more complicated, and its description is captured in the following lemma.

\begin{lem}
There is an isomorphism
$$
	W_{H_1, H_2} \cong \left(\left(W_{H_1, H_2}^0 \times G\right) \sslash H_1 \right) \sslash N_{H_1, H_2},
$$
where $H_1$ acts on $W_{H_1, H_2}^0 \times G$ by $h \cdot (A, B, F_i, g) = (A, B, h F_ih^{-1}, gh^{-1})$,
and $N_{H_1, H_2} = (N_{G}(H_1) \cap N_{G}(H_2)) / H_1$. % is the Weyl group of $H_1$ in $G$, and 
\end{lem}

\begin{proof}
We start with the map 
$$
	W_{H_1, H_2}^0 \times G \to W_{H_1, H_2}, \qquad (A, B, F_i, g) \mapsto (gAg^{-1}, gBg^{-1}, gF_ig^{-1}).
$$
With the definition of the $H_1$-action given above, this map is $H_1$-invariant, so it descends to a morphism 
$$
	\left(W_{H_1, H_2}^0 \times G\right) \sslash H_1 \to W_{H_1, H_2}.
$$
The failure of the injectivity of this morphism is precisely given by the action of the %Weyl 
group $N_{H_1, H_2}$ so, taking the quotient by this group, we get the desired isomorphism.
\end{proof}

Now, using the fact that $W_{H_1, H_2}^0 = V_{H_1, H_2}^0 \times H_2^{d-1}$ and the action of $H_1$ is trivial on $V_{H_1, H_2}^0$, we get that
\begin{equation}\label{eq:iso-w}
	W_{H_1, H_2} \cong \left(\left(W_{H_1, H_2}^0 \times G\right) \sslash H_1 \right) \sslash N_{H_1, H_2} = \left(V_{H_1, H_2}^0 \times \left(H_2^{d-1} \times G\right) \sslash H_1 \right) \sslash N_{H_1, H_2}.
\end{equation}

Regarding the quotient by $G$ on $W_{H_1, H_2}$ by conjugation, we note that under the isomorphism 
(\ref{eq:iso-w}) the group $G$ acts only on the factor $G$ on the left. Hence, we get an isomorphism
$$
	\mathcal{W}_{H_1, H_2} := W_{H_1, H_2} \sslash G \cong \left(V_{H_1, H_2}^0 \times H_2^{d-1} \sslash H_1 \right) \sslash N_{H_1, H_2}.
$$
Summarizing these observations, and setting $\widetilde{W}_{H_1, H_2} = V_{H_1, H_2}^0 \times \left(H_2^{d-1} \times G\right) \sslash H_1$ and $\widetilde{\mathcal{W}}_{H_1, H_1} = V_{H_1, H_2}^0 \times H_2^{d-1} \sslash H_1$, we get the following result.

\begin{thm}\label{thm:local-model}
Let $H_2$ be a subgroup of $G$ and let $H_1 \triangleleft H_2$ be a normal subgroup for which the stratum $W_{H_1, H_2}$ is non-empty. Set $N_{H_1, H_2} = (N_{G}(H_1) \cap N_G(H_2)) / H_1$. We get a commutative diagram
$$
\xymatrix{
\widetilde{W}_{H_1, H_2} \ar[d] \ar[r] &W_{H_1, H_2} \ar[d]  \subset \R^d(n,m)\\
\widetilde{\mathcal{W}}_{H_1, H_2} \ar[r] & \mathcal{W}_{H_1, H_2} \subset \X^d(n,m) 
}
$$
where the right-most arrow is the quotient by the adjoint action of $G$, and:
\begin{enumerate}
	\item We have an identification $\widetilde{W}_{H_1, H_2} = V_{H_1, H_2}^0 \times \left(H_2^{d-1} \times G\right) \sslash H_1$, with action of $H_1$ on $H_2^{d-1} \times G$ given by $ h \cdot (F_i, g) = (hF_ih^{-1}, gh^{-1})$.
	\item With this description, the map $\widetilde{W}_{H_1, H_2} \to W_{H_1, H_2}$ is $$(A,B, [F_i, g]_{H_1} ) \mapsto (gAg^{-1}, gBg^{-1}, gF_ig^{-1}).$$
This map is equivalent to the quotient map under the action of the group $N_{H_1, H_2} = (N_G(H_1) \cap N_G(H_2)) / H_1$ on 
$\widetilde{W}_{H_1, H_2}$ by $n \cdot (A, B, [F_i, g]_{H_1}) = (n A n^{-1}, n B n^{-1}, [nF_i n^{-1}, gn^{-1}]_{H_1})$.
	\item The space $\widetilde{\mathcal{W}}_{H_1, H_2}=V_{H_1, H_2}^0 \times H_2^{d-1} \sslash H_1$. 
	The vertical map $\widetilde{W}_{H_1, H_2} \to \widetilde{\mathcal{W}}_{H_1, H_2}$ is the GIT quotient map for the induced $G$-action.
	\item The bottom arrow is the quotient map for the induced action of $N_{H_1, H_2}$ on $\widetilde{\mathcal{W}}_{H_1, H_2}=V_{H_1, H_2}^0 \times H_2^{d-1} \sslash H_1$.
\end{enumerate}
\end{thm}

%\begin{rem} {\color{red} VICENTE: Yo creo que esto es falso. Para que lo queremos? ÁNGEL: Intentaba ser una justificación de que la acción está bien definidal. Pero quizá podemos eliminar el remark directamente, porque creo que nunca lo usamos explícitamente. \color{black}}
%Let us define the action of $N_{H_1, H_2}$ on $V_{H_1, H_2}^0 \times H_2^{d-1} \times G$ by
%$$
%	n \cdot (A, B, F_i, g) = (n A n^{-1}, n B n^{-1}, nF_i n^{-1}, gn^{-1}),
%$$
%and consider the previous action of $H_1$ on the same space. Since the elements of $N_{H_1, H_2}$ are (classes of) elements of the normalizer of $H_1$, for any $n \in N_{H_1, H_2}$ and $h \in H_1$ we have that $n \cdot (h \cdot -) = h' \cdot (n \cdot -)$ for a certain $h' \in H_1$. In particular, this implies that the action of $N_{H_1, H_2}$ on $\widetilde{\mathcal{W}}_{H_1, H_2}$ is well defined and that the quotients $\left(V_{H_1, H_2}^0 \times (H_2^{d-1} \times G) \sslash H_1 \right) \sslash N_{H_1, H_2}$ and $\left(\left(V_{H_1, H_2}^0 \times (H_2^{d-1} \times G \right) \sslash N_{H_1, H_2} \right) \sslash H_1$ are isomorphic.
%\end{rem}

\begin{rem}
When computing the spaces $W_{H_1, H_2} \subset \R^d(n,m)$, it will be convenient to analyze the strata $V_{H_1, H_2} \subset \R(n,m)$ first, and their respective quotients. In the following, the corresponding quotients by $G$ will also be denoted by $\mathcal{V}_{H_1, H_2} = V_{H_1, H_2} \sslash G$ and $\widetilde{\mathcal{V}}_{H_1, H_2} = \widetilde{V}_{H_1, H_2} \sslash G$. Notice that, by Proposition \ref{prop:structure-character-var}, we have a natural identification $\widetilde{\mathcal{V}}_{H_1, H_2} \cong V_{H_1, H_2}^0$.
\end{rem}

\begin{cor}\label{cor:quotient-connected}
Suppose that $N_{H_1, H_2}$ is a connected algebraic group. Then we have an equality of $E$-polynomials
$$
	e\left({\mathcal{W}}_{H_1, H_2}\right) = e\left(\mathcal{V}_{H_1, H_2}\right) e\left(H_2^{d-1} \sslash H_1\right).
$$
\begin{proof}
Now, the action of $N_{H_1, H_2}$ on $ \widetilde{\mathcal{W}}_{H_1, H_2} = V^0_{H_1, H_2} \times H_2^{d-1} \sslash H_1$ is free, so the map $\widetilde{\mathcal{W}}_{H_1, H_2} \to {\mathcal{W}}_{H_1, H_2}$ is a principal $N_{H_1, H_2}$-bundle (in the analytic topology), and analogously for $\mathcal{V}_{H_1, H_2}^0 = \widetilde{\mathcal{V}}_{H_1, H_2} \to \mathcal{V}_{H_1, H_2}$. Then, by \cite[Remark 2.5]{LMN}, we have that $e\left({\mathcal{W}}_{H_1, H_2}\right) = e(\widetilde{\mathcal{W}}_{H_1, H_2})/e(N_{H_1, H_2})$ and $e({\mathcal{V}}_{H_1, H_2}) = e\left({V}_{H_1, H_2}^0\right)/e(N_{H_1, H_2})$, and thus
$$
	e\left({\mathcal{W}}_{H_1, H_2}\right) = \frac{e\left(V^0_{H_1, H_2}\right)}{e(N_{H_1, H_2})} e\left(H_2^{d-1} \sslash H_1\right) = e\left(\mathcal{V}_{H_1, H_2}\right) e\left(H_2^{d-1} \sslash H_1\right).
$$
\end{proof}

\end{cor}

The previous result is particularly interesting in the following case. Let $G =\GL_r(\CC)$ and denote by $\R(n,m)^*$ and $\R^d(n,m)^*$ the subvarieties of irreducible representations of $\R(n,m)$ and $\R^d(n,m)$, respectively. Recall the projection map of (\ref{eq:proj-knot-link})
$$
	\pi: \R^d(n,m) \to \R(n,m), \qquad \pi(A,B,F_i) = (A,B).
$$
We trivially have that $\pi^{-1}(\R(n,m)^*)  \subset \R^d(n,m)^*$. Furthermore, by Schur's lemma, for any $\rho = (A,B) \in \R(n,m)^*$, we have that $A^n = B^m$ is a multiple of the identity and the stabilizer of $\rho$ are multiples of the identity. In particular, this means that $\R(n,m)^* \subset V_{\CC^{\ast}\Id, G}$. Hence, in this case $H_1 = \CC^* \Id$ and $H_2 = G$, so $N_{H_1, H_2} = \GL_r( \CC) / \CC^* \Id = \PGL_r( \CC)$, which is a connected group and thus
$$
	e(\cW_{\CC^*\Id, G}) = e(\mathcal{V}_{\CC^*\Id, G}) e(G)^{d-1},
$$
where we have used that $G^{d-1} \sslash \CC^* \Id = G^{d-1}$ since $\CC^*\Id$ acts trivially on $G^{d-1}$ by conjugation.
Furthermore, when $G=\SL_r(\CC)$, we have $\R(n,m)^* \subset V_{\bm{\mu}_r \Id, G}$, where $\bm{\mu}_r$ denotes the set of $r$-th roots of unity and the rest of the argument works verbatim. We summarize these observations in the following.

\begin{cor} \label{cor:irreduciblelocus}
Let $G=\GL_r(\CC)$ or $\SL_r(\CC)$. Over the irreducible locus $\X^d(n,m)^*$ of the character variety, we have an equality
$$
	e\left(\pi^{-1}(\X(n,m)^*)\right) = e(\X(n,m)^*)e(G)^{d-1}.
$$
\end{cor}

%%%%%%%%%%%%%%%%%%%%%%%%%%%%%%%%%%%%%%%%%%%%%%%%
\section{$\sldos$-character varieties of torus links} \label{sec:sl2geodesc}
%%%%%%%%%%%%%%%%%%%%%%%%%%%%%%%%%%%%%%%%%%%%%%%%

Let us consider the case $G=\sldos$, that we abbreviate as $\SL_2$. 
We aim to describe the character variety:
$$
\X_{\SL_2}^d(n,m)=\lbrace(A,B,F_1,\ldots,F_{d-1})\in \SL_2^{d+1} \mid A^n=B^m, \: [A^n,F_i]=\Id \rbrace \sslash \SL_2.
$$
For simplicity, we shall drop the subscript from the notation and we will denote it simply by $\X^d(n,m)$. In this section, let us write
\begin{equation}\label{def:subgroupsSL2}
H_1 := \bm{\mu}_2 \Id =\lbrace \pm \Id \rbrace, \quad
H_2 := \left\lbrace \begin{pmatrix} \lambda & 0 \\ 0 & \lambda^{-1} \end{pmatrix} | \, \lambda \in \mathbb{C}^{\ast} \right\rbrace \cong \CC^{\ast},
\end{equation}
which are the stabilizers in $\sldos$ of $\GL_2(\CC)$ and $ \GL_1(\CC)\times\GL_1(\CC)$, respectively. The associated groups are $N_{H_1, H_1} = \PGL_2( \CC)$ and $N_{H_2, H_2} = (\CC^* \ltimes \ZZ_2)/\CC^* = \ZZ_2$, where the latter acts on $H_2$ by permutation of eigenvalues.

In the case $d=1$, the character variety $\X^1(n,m) = \X(n,m)$ for torus knots was studied in \cite{mun:2009} for coprime $n,m$. It can be geometrically described as follows:
\begin{itemize}
\item $\X(n,m)^*$, the stratum of irreducible representations, is a collection of $(m-1)(n-1)/2$ components isomorphic to $\CC-\lbrace 0,1\rbrace$, parametrized by roots of unity. The closure of each component intersects the reducible stratum at two points. By Schur's lemma, these points have $H_1$ as stabilizer and, for any $(A,B) \in \X(n,m)^*$, we have $A^n = \pm \Id$. Indeed, $\X(n,m)^* = \cU_{H_1} = \cV_{H_1,\SL_2}$.
\item $\X(n,m)^{\textrm{TR}}$, the stratum of (totally) reducible representations. It is isomorphic to $\CC^{\ast}/ \ZZ_2$, with action $t \mapsto t^{-1}, \: t\in \CC^{\ast}$. The space is isomorphic to $\CC$, parametrized by $s=t+t^{-1}$. For the stabilizer, there are two options:
\begin{itemize}
	\item If either $A,B\not\in H_1$, then these points have stabilizer $H_2$ and $\cU_{H_2} = \X(n,m)^{\textrm{TR}} - \{(\pm \Id, \pm \Id)\}$. Note
	that the condition $A^n=B^m$, means that if $n,m$ are both odd, then we remove $S=\{(\Id,\Id),(-\Id,-\Id)\}$, and if $n$ is even and $m$ is
	odd then we remove $S=\{(\Id,\Id), (-\Id,\Id)\}$. % (and in reverse order if $n$ odd and $m$ even). 
	In all cases, $\cU_{H_2} = \X(n,m)^{\textrm{TR}} - S$, with $S$ being two points.

	The stabilizer of $A^n$ for a representation varies depending on whether $A^n = \pm \Id$ or not. 
\begin{itemize}
	\item We have $A^n = \pm \Id\in H_1$ if and only if $A$ is conjugated to $\diag(\epsilon, \epsilon^{-1})$, with $\epsilon \in \bm{\mu}_{2n}$ a $2n$-th root of unity. Hence, we get that $V_{H_2, \SL_2}$ is the collection of $(A, B) \in \R(n,m) - \R(n,m)^*$ such that $A \sim \diag(\epsilon, \epsilon^{-1})$  and $B \sim \diag(\varepsilon, \varepsilon^{-1})$ for $(\epsilon , \varepsilon) \in \bm{\mu}_{2n} \times \bm{\mu}_{2m} - \{(\pm 1, \pm 1)\}$.
	The condition $\epsilon^n=\varepsilon^m$ translates to $\epsilon=\upsilon^m$, $\varepsilon=\upsilon^n$, 
	with $\upsilon \in \bm{\mu}_{2mn}-\{\pm 1\}$. Hence 
	$\widetilde{\cV}_{H_2, \SL_2} = \bm{\mu}_{2mn} - \{\pm 1\}$ and $\cV_{H_2, \SL_2} = \widetilde{\cV}_{H_2, \SL_2}/\ZZ_2 $ 
	with the action $\upsilon \sim \upsilon^{-1}$. The resulting space has $mn-1$ points.
	\item In the remaining cases, since $A$ is diagonalizable, so is $A^n$, with different eigenvalues, so it belongs to $H_2$. Hence, the stabilizer is conjugated to $H_2$ and we get the stratum $\cV_{H_2, H_2}$. Looking at the parameter $t$, we have $\widetilde{\cV}_{H_2, H_2} = \CC^*-\bm{\mu}_{2mn}$ and 
	$\cV_{H_2, H_2} = \widetilde{\cV}_{H_2, H_2} / \ZZ_2$, which is isomorphic to $\CC$ with $mn+1$ points removed.
\end{itemize}
\item If $A,B\in H_1$, all stabilizers are equal to $\SL_2$, so that $\cU_{\SL_2}=S \subset \lbrace (\pm \Id, \pm \Id) \rbrace$, described above,
 and $\widetilde{\cV}_{\SL_2,\SL_2}=\cV_{\SL_2,\SL_2}=S$ since $N_{\SL_2} = 1$.
\end{itemize}
\end{itemize}

Having analyzed these strata, let us look at the corresponding character variety $\X^d(n,m)$. 
\begin{itemize}
\item For $\pi^{-1}(\X(n,m)^*) = \cW_{H_1, \SL_2}$. %Corollary \ref{cor:irreduciblelocus} shows that $\cW_{H_1, \SL_2}= \X(n,m)^* \times \SL_2^{d-1}$.
\item For $\cW_{\SL_2,\SL_2}$, recall that $V^0_{\SL_2,\SL_2} = S\subset \{(\pm \Id, \pm \Id)\}$ and we have that
\begin{align*}
	\cW_{\SL_2,\SL_2} &= \left( S \times \left( \SL_2^{d-1} \right) \sslash \SL_2\right) \sslash N_{\SL_2} = S \times \left( \SL_2^{d-1} \sslash \SL_2 \right).
\end{align*}
This stratum consists of two copies of the representation variety of the free group in $d-1$ generators.
\item For $\widetilde{\cW}_{H_2, \SL_2}$, we have $V^0_{H_2, \SL_2} = \bm{\mu}_{2mn} - \{\pm 1\}$ with the $\ZZ_2$-action $\upsilon \mapsto 
\upsilon^{-1}$. Hence, recalling from \eqref{def:subgroupsSL2} that $H_2\cong \CC^*$, we get
\begin{align*}
	\widetilde{\cW}_{H_2,\SL_2} &= \left(\bm{\mu}_{2mn} - \{\pm 1 \}\right) \times \SL_2^{d-1} \sslash \CC^*,
\end{align*}
and $\cW_{H_2,\SL_2} = \widetilde{\cW}_{H_2,\SL_2} / \ZZ_2$. The $\CC^{\ast}$-action on $\SL_2^{d-1} $ is given by conjugation by diagonal matrices, so it may be trivial or not depending on the reducibility of the representation. 

\item For $\widetilde{\cW}_{H_2, H_2}$, since $V^0_{H_2, H_2} = \CC^*-\bm{\mu}_{2mn}$ and $H_2\cong \CC^{\ast}$, we have
\begin{align*}
	\widetilde{\cW}_{H_2,H_2} &=  \left(\CC^*-\bm{\mu}_{2mn}\right) \times (\CC^*)^{d-1} \sslash \CC^* = \left(\CC^*-\bm{\mu}_{2mn}\right) \times (\CC^*)^{d-1},
\end{align*}
where in the last equality we used that the $H_2$-action on $(\CC^*)^{d-1}$ is trivial. Hence, we have ${\cW}_{H_2,H_2} = \left(\left(\CC^*-\bm{\mu}_{2mn}\right) \times (\CC^*)^{d-1}\right) / \ZZ_2$.
\end{itemize}

\subsection{$E$-polynomials of $\X^d_{\SL_2}(n,m)$}

The main result of this section is the following theorem.

\begin{thm}\label{thm:e-poly-sl2}
The $E$-polynomial of $\X^d_{\SL_2}(n,m)$ is given by:
\begin{align*}
e(\X^d_{\SL_2}(n,m)) = & \phantom{+} \frac{1}{2}(m-1)(n-1)(q-2)(q^3-q)^{d-1} \\
& +(mn-1)\left((q^2+q)(q^3-q)^{d-2}-(q-1)^{d-2}(2q^{d-1}-1) \right) \\
& + 2\left( (q^3-q)^{d-2}-(q^2-q)^{d-2}+ \frac{1}{2}q((q+1)^{d-2}+(q-1)^{d-2}) \right) \\
& + (q-2)\frac{1}{2}\left( (q+1)^{d-1}+(q-1)^{d-1} \right) +\frac{1}{2}\left( (q+1)^{d-1}-(q-1)^{d-1} \right).
\end{align*}
\end{thm}

\begin{proof}
We follow the decomposition given in Section \ref{sec:sl2geodesc},
$$
\X_{\SL_2}^d(n,m)= \cW_{H_1, \SL_2} \sqcup \cW_{H_2,\SL_2} \sqcup \cW_{\SL_2,\SL_2} \sqcup \cW_{H_2,H_2},
$$
that stratifies the character variety into locally closed subvarieties that are well suited for computations.
First of all,
$$
e(\cW_{H_1,\SL_2})=\frac{1}{2}(m-1)(n-1)(q-2)(q^3-q)^{d-1},
$$
since $e(\cW_{H_1, \SL_2}) = e(\X(n,m)^{\ast}) e(\sldos)^{d-1}$, and $\X(n,m)^{\ast}$ consists of $\frac{1}{2}(m-1)(n-1)$ components isomorphic to $\CC$ with two points removed.

Secondly, since $\cW_{\SL_2,\SL_2} \cong S \times \sldos^{d-1} \sslash \sldos$, the stratum is given by two copies of the character variety of the free group with $d-1$ generators. Its $E$-polynomial was computed in \cite{cavazos-lawton:2014}, so we deduce that
\begin{align*}
e(\cW_{\SL_2,\SL_2})& =2\left((q-1)^{d-2}((q+1)^{d-2}-1)q^{d-2} + \frac{1}{2}q\left( (q-1)^{d-2}+(q+1)^{d-2} \right) \right) \\
& = 2 \left( (q^3-q)^{d-2}-(q^2-q)^{d-2}+ \frac{1}{2}q((q+1)^{d-2}+(q-1)^{d-2}) \right).
\end{align*}

If we look now at $\cW_{H_2,\SL_2}$, there are $2mn-2$ pairs $(\epsilon,\varepsilon) \in \bm{\mu}_{2n}\times \bm{\mu}_{2m}$ that satisfy $\epsilon^n=\varepsilon^m=\pm 1$, which decomposes $\cW_{H_2,\SL_2}$ into several disjoint components. The $\ZZ_2$-action identifies them and leaves $mn-1$ components, each one isomorphic to $\sldos^{d-1}\sslash \CC^{\ast}$. Therefore, the fibre over $(\epsilon,\varepsilon)$ is given by $d-1$ elements in $\sldos$, quotiented by the $\CC^{\ast}$-action by conjugation by diagonal matrices. Call $\cW^f_{H_2,\SL_2}$ one of these fibers,
so that $e(\cW_{H_2,\SL_2})=(mn-1) e(\cW^f_{H_2,\SL_2})$.

To compute $e(\cW^f_{H_2,\SL_2})$, we do as follows.
Reducible representations occur when all the $F_i$ are simultaneously upper or lower triangular, which are parametrized by two copies of $(\CC^{\ast})^{d-1}\times \CC^{d-1}$, where the diagonal matrices $(\CC^{\ast})$ have been accounted twice. Substracting its contribution, the $E$-polynomial of the reducible part of $W_{H_2,\SL_2}$ is given by
$$
e(W_{H_2,\SL_2}^{f,\textrm{red}})=2(q-1)^{d-1}q^{d-1}-(q-1)^{d-1}=(q-1)^{d-1}(2q^{d-1}-1).
$$
We obtain the $E$-polynomial of the irreducible locus as
\begin{align*}
e(W^{f,\textrm{irr}}_{H_2,\SL_2}) & = e\left(\sldos^{d-1}\right)-e(W^{\textrm{red}}_{H_2,\SL_2})= (q^3-q)^{d-1}-(q-1)^{d-1}(2q^{d-1}-1),
\end{align*}
and since the action is free,
$$
e(\cW^{f,\textrm{irr}}_{H_2,\SL_2})  = e(W^{f,\textrm{irr}}_{H_2,\SL_2})/(q-1) 
= (q^2+q)(q^3-q)^{d-2}-(q-1)^{d-2}(2q^{d-1}-1).
$$

All the remaining representations that need to be taken into account are reducible, and they belong to either  $\mathcal{W}_{H_2,\SL_2}^{\textrm{red}}$ or $\cW_{H_2,H_2}$. They are all $S$-equivalent to
$$
(A,B,F_i)\sim \left( \begin{pmatrix} t^m & 0 \\ 0 & t^{-m}   \end{pmatrix}, \begin{pmatrix} t^n & 0 \\ 0 & t^{-n} \end{pmatrix}, \begin{pmatrix} a_i & 0 \\ 0 & a_i^{-1} \end{pmatrix} \right), \: i=1,\ldots,d-1,
$$
where $t\neq 0,\pm 1$, $a_i\in \CC^{\ast}$, and the $\ZZ_2$-action is given now by $(t,a_1,\ldots,a_{d-1})\sim (t^{-1},a_1^{-1},\ldots,a_{d-1}^{-1})$. Denoting $B=\CC^{\ast}-\lbrace \pm 1 \rbrace$ and $F=(\CC^{\ast})^{d-1}$, and using the notation of Example \ref{ex:pm-formula}, we have that $
e(B^+)=q-2, \: e(B^-)=1 $ and also
$$
e(F^+)=\frac{1}{2}\left((q+1)^{d-1}+(q-1)^{d-1} \right), \quad e(F^-)=\frac{1}{2}\left((q+1)^{d-1}- (q-1)^{d-1} \right).
$$
Hence, formula (\ref{eqn:epolyZ2}) implies that
\begin{align*}
e(W_{H_2,\SL_2}^{\textrm{red}} \sslash \CC^*)+e(\cW_{H_2,H_2}) & = e(B^+)e(F^+)+e(B^-)e(F^-) \\ 
 & = (q-2) \frac12 \left((q+1)^{d-1}+(q-1)^{d-1} \right) + \frac{1}{2}\left((q+1)^{d-1} - (q-1)^{d-1} \right).
\end{align*}
Adding all contributions yields the desired polynomial.
\end{proof}

\begin{rem}
Note that when $d=2$ and $m=1$, we get that
$$
e(\X^2_{\SL_2}(1,n)) = (n-1)(q^2-q+1)+q^2+1,
$$
which agrees with the $E$-polynomial of the Hopf link that was computed in \cite{gonmun:2022}. We also obtain when $d=1$ that
$$
e(\X^1_{\SL_2}(n,m)) = \frac{1}{2}(m-1)(n-1)(q-2)+q,
$$
which agrees with the $E$-polynomial of the torus knot character variety described in \cite{mun:2009}.
\end{rem}

%%%%%%%%%%%%%%%%%%%%%%%%%%%%%%%%%%%%%%%%
\section{$\lambda$-character varieties}\label{sec:lambda-character-varieties}
%%%%%%%%%%%%%%%%%%%%%%%%%%%%%%%%%%%%%%%%

Before proceeding with the rank $3$ case, let us discuss a new kind of character varieties that will be very useful for upcoming calculations.
Let $\Gamma$ be a finitely generated group, $V$ a finite dimensional complex vector space and let us consider $\mathcal{R}_{\GL(V)}(\Gamma) = \Hom(\Gamma, \GL(V))$ the $\GL(V)$-representation variety. Fix a decomposition $V_\bullet = V_1 \oplus V_2 \ldots \oplus V_s$ into non-trivial vector subspaces $V_i \subset V$ and set $\GL(V_\bullet) = \GL(V_1) \times \GL(V_2) \times \ldots \times \GL(V_s)$. In this context, the $V_\bullet$-character variety is the quotient space
$$
	\mathcal{X}_{V_\bullet}(\Gamma) = \mathcal{R}_{\GL(V)}(\Gamma) \sslash \GL(V_\bullet).
$$
Here $\GL(V_\bullet) < \GL(V)$ is acting on $\mathcal{R}_{\GL(V)}(\Gamma)$ by conjugation.

\begin{ex}
In the trivial case $V_\bullet = V$, we have $\GL(V_\bullet) = \GL(V)$ and we recover the usual $\GL(V)$-character variety.
\end{ex}

Many of the usual definitions for representations varieties can be extended to the $V_\bullet$-context. A $V_\bullet$-subspace is a linear subspace $W \subset V$ such that $W = W_1 \oplus W_2 \oplus \ldots \oplus W_s$ with $W_i \subset V_i$. A representation $\Gamma \to \GL(V)$ is said to be $V_\bullet$-irreducible if it admits no non-trivial invariant $V_\bullet$-subspace. In the same vein, a representation $\rho$ is said to be $V_\bullet$-semisimple if there exist invariant $V_\bullet$-subspaces $W^1, \ldots, W^r$ such that $V = W^1 \oplus \ldots \oplus W^r$ and $\rho|_{W^i}$ is $V_\bullet$-irreducible for all $i = 1, \ldots, r$.

First of all, observe that we can straightforwardly adapt Schur's lemma to this context.
 
\begin{lem}
Let $\rho \in  \mathcal{R}_{\GL(V)}(\Gamma)$ be a $V_\bullet$-irreducible representation and $A \in \GL(V_\bullet)$ an intertwining map, i.e.\ $A \circ \rho(\gamma) = \rho(\gamma) \circ A$, for all $\gamma \in \Gamma$. Then $A$ is a multiple of the identity.

\begin{proof}
Let $\lambda \in \CC^*$ be an eigenvalue of $A$. The map $A - \lambda \Id$ is also an intertwining map in $\GL(V_\bullet)$, so its kernel is a non-zero invariant $V_\bullet$-subspace. Since $\rho$ is $V_\bullet$-irreducible, it admits no non-trivial $V_\bullet$-subspaces, so the kernel of $A - \lambda \Id$ must be the whole space $V$, and thus $A = \lambda \Id$, as claimed.
\end{proof}
\end{lem}

Additionally, we can obtain an analogous criterion of polystability for the $V_\bullet$-character variety. The proof is completely analogous to the one provided in \cite[Section 6.1]{gon:2023} using only $V_\bullet$-conjugation in Mumford's criterion.

\begin{prop}\label{prop-polystable-Vlambda}
Consider the action of $\GL(V_\bullet)$ on $\mathcal{R}_{\GL(V)}(\Gamma)$. A representation is polystable if and only if it is $V_\bullet$-semisimple.
\end{prop}

\begin{cor}
Let $\mathcal{R}_{\GL(V)}^{\lambda-\textrm{ss}}(\Gamma) \subset \mathcal{R}_{\GL(V)}(\Gamma)$ be 
the subspace of $V_\bullet$-semisimple representations (in other words, its polystable locus). The $E$-polynomial of the $V_\bullet$-character variety agrees with the one of the orbit space $\mathcal{R}_{\GL(V)}^{\lambda-\textrm{ss}}(\Gamma) / \GL(V_\bullet)$.
\end{cor}

A particular case that we shall use in the following is the case in which $V = \CC^n$ and the decomposition is induced from an ordered partition $\lambda = (\lambda_1, \lambda_2, \ldots, \lambda_s)$ of $n$, i.e.\ satisfying $\sum_i \lambda_i = n$. In this case, if $e_1, \ldots , e_n$ denotes the canonical basis of $\CC^n$, we can set $V_1 = \langle e_1, \ldots, e_{\lambda_1}\rangle$, $V_2 = \langle e_{\lambda_1 +1}, \ldots, e_{\lambda_1 + \lambda_2}\rangle$ and, in general,
$V_i = \langle e_{\lambda_1+ \ldots+\lambda_{i-1} + 1}, \ldots, e_{ \lambda_1 +\ldots + \lambda_i}\rangle$.
%where the sums run over the elements $\lambda_j$ with $j < i$.

For simplicity, we shall denote by $\GL_\lambda$ the group $\GL(V_\bullet) = \GL_{\lambda_1} \times \GL_{\lambda_2} \times \ldots \times \GL_{\lambda_s}$ and, in general, to emphasize the role of the partition we shall talk about $\lambda$-character variety, $\lambda$-subspaces and $\lambda$-semisimplicity instead of the corresponding $V_\bullet$-terms. In particular, the $\lambda$-character variety is the quotient space
$$
	\mathcal{X}_\lambda(\Gamma) = \mathcal{R}_{\GL_n}(\Gamma) \sslash \GL_\lambda,
$$
where $\GL_n = \GL_n(\CC)$.

\begin{rem}
Similar considerations can be done for $\SL_n$-representations in many cases. For instance, for the free group $\Gamma = \Free{d}$ in $d$ generators, recall that we have natural fibrations
$$
	(\CC^*)^d \to \mathcal{R}_{\GL_n}(\Free{d})  \to \mathcal{R}_{\SL_n}(\Free{d}),
$$
given by re-scaling the first column of each matrix. These fibrations descend to the quotient under the action of $\GL_\lambda$, so we have that
$$
	e\left( \mathcal{R}_{\SL_n}(\Free{d}) \sslash \GL_\lambda \right) = \frac{1}{(q-1)^d} e(\mathcal{X}_{\lambda}(\Free{d})).
$$
\end{rem}

%%%%%%%%%%%%%%%%%%%%%%%%%%%%%%%%%
\subsection{The case $\lambda = (1, 1, \ldots, 1)$}
%%%%%%%%%%%%%%%%%%%%%%%%%%%%%%%%%

In this section, we shall study in detail the case $\lambda =  (1, 1, \ldots, 1)$. For this case, we have $\GL_\lambda = \CC^* \times \ldots \times \CC^*$ with action by re-scaling. To be precise, given $(A_1, \ldots, A_r) \in \mathcal{R}_{\GL_n}(\Gamma)$ with $A_k = (a_{ij}^k)$, and $t = (t_1, \ldots, t_n) \in (\CC^*)^n$, we have
\begin{equation}\label{eq:action-cstar}
	t \cdot (A_1, \ldots, A_r) = \left((a_{ij}^1 t_it_j^{-1}), \ldots, (a_{ij}^r t_it_j^{-1}) \right).
\end{equation}
There is a natural stratification of the $\lambda$-semisimple representations $\mathcal{R}_{\GL_n}^{\lambda-\textrm{ss}}(\Gamma) \subset \mathcal{R}_{\GL_n}(\Gamma)$. Let $\mu$ be a partition of the set $\{1, \ldots, n\}$, that is $\mu$ is a disjoint collection of sets of the form $\mu_i = \{i_1, \ldots, i_m\}$ with $1 \leq i_j \leq n$. Then $\mathcal{R}_{\GL_n}^{\mu}(\Gamma) \subset \mathcal{R}_{\GL_n}^{\lambda-\textrm{ss}}(\Gamma)$ is the collection of $\lambda$-semisimple representations with a decomposition into $\lambda$-irreducible representations of the form $W_1 \oplus \ldots \oplus W_s$ where $W_i = \langle e_{i_1}, \ldots, e_{i_m}\rangle$. In particular, the $\lambda$-irreducible representations correspond to the trivial partition $\mu = \{\{1, \ldots, n\}\}$.

In this way, we have a natural stratification
$$
	\mathcal{R}_{\GL_n}^{\lambda-\textrm{ss}}(\Gamma) / (\CC^*)^n = \bigsqcup_{\mu} \mathcal{R}_{\GL_n}^{\mu}(\Gamma) / (\CC^*)^n,
$$
where the disjoint union runs over the partitions $\mu$ of the set $\{1, \ldots, n\}$.

Observe that the diagonal matrices $\CC^* < (\CC^*)^n$ act trivially on $\mathcal{R}_{\GL_n}^{\lambda-\textrm{ss}}(\Gamma)$ by (\ref{eq:action-cstar}). However, if we mod out by this subgroup, it turns out the the action is generically free.

\begin{prop}
If $\mu = \{\mu_1, \ldots, \mu_r\}$ is a partition with $|\mu_i| > 1$ for all $i$, then the action of $(\CC^*)^n/\CC^*$ on $\mathcal{R}_{\GL_n}^{\mu}(\Gamma)$ is free.
\begin{proof}
Let $(A_1, \ldots, A_r) \in \mathcal{R}_{\GL_n}(\Gamma)$. Since $|\mu_i| > 1$ for all $i$, each row and column has an off-diagonal non-vanishing entry in some of the matrices $A_1, \ldots, A_r$. Hence, every $t = (t_1, \ldots, t_n) \in (\CC^*)^n$ acts non-trivially except when $t_i = t_j$ for all $i,j$, which is exactly the condition $t \in \CC^* < (\CC^*)^n$.
\end{proof}
\end{prop}

Recall that if $X/G$ is an orbit space GIT quotient with $G$ connected and acting freely, then $e(X/G) = e(X)/e(G)$, c.f.\ \cite[Remark 2.5]{LMN}. In this way, let us denote by $\mathcal{P}_0$ the collection of partitions $\mu = \{\mu_i\}$ of $\{1, \ldots, n\}$ with $|\mu_i| > 1$ for all $i$. We have that 
$$
	e(\mathcal{X}_\lambda(\Gamma)) = e(\mathcal{R}_{\GL_n}^{\lambda-\textrm{ss}}(\Gamma) / (\CC^*)^n) = \sum_{\mu \not\in \mathcal{P}_0} e\left(\mathcal{R}_{\GL_n}^{\mu}(\Gamma) / (\CC^*)^n\right) +  \frac{1}{(q-1)^{n-1}}\sum_{\mu \in \mathcal{P}_0} e(\mathcal{R}_{\GL_n}^{\mu}(\Gamma)).
$$
The two terms of this sum can be computed using a different strategy. For each summand of the first term, we have that $\mu = \{i\} \cup \mu'$ for some partition $\mu'$ and $\mathcal{R}_{\GL_n}^{\mu}(\Gamma)/(\CC^*)^{n} = \CC^* \times \mathcal{R}_{\GL_{n-1}}^{\mu'}(\Gamma)/(\CC^*)^{n-1}$, so they can be computed recursively on $n$. For the second term, it is typically easier to compute its complement using the inclusion-exclusion principle, since it is made of $\lambda$-reducible representations.

\begin{ex}
Let us take $\Gamma = \Free{d}$ the free group in $d$ generators. In the case $n = 1$ we have $\mathcal{R}_{\GL_1}(\Gamma) = (\CC^*)^d$ with the trivial action, so $e(\mathcal{X}_{(1)}(\Gamma)) = (q-1)^d$.
\end{ex}

\begin{ex}\label{ex:totally-action-sl2}
In the case $n = 2$ and $\Gamma =  \Free{d}$, the possible partitions are $\mu^1 = \{\{1\}, \{2\}\}$ and $\mu^2 = \{\{1,2\}\}$, and $\mathcal{P}_0 = \{\mu^2\}$. Hence, in general we have
\begin{equation}\label{eq:decomposition-X11}
	e(\mathcal{X}_{(1,1)}(\Gamma)) = e\left(\mathcal{R}_{\GL_2}^{\mu^1}(\Gamma) / (\CC^*)^2\right) +  \frac{1}{q-1} e(\mathcal{R}_{\GL_2}^{\mu^2}(\Gamma)).
\end{equation}
On the one hand, we have $\mathcal{R}_{\GL_2}^{\mu^1}(\Gamma) = \mathcal{R}_{\GL_1}(\Gamma) \times \mathcal{R}_{\GL_1}(\Gamma)$ with the trivial action, so $e\left(\mathcal{R}_{\GL_2}^{\mu^1}(\Gamma) / (\CC^*)^2\right) = e(\mathcal{R}_{\GL_1}(\Gamma))^2$. In particular, for $\Gamma =  \Free{d}$, we have $e\left(\mathcal{R}_{\GL_2}^{\mu^1}(\Gamma) / (\CC^*)^{2}\right) = (q-1)^{2d}$. On the other hand, the space $\mathcal{R}_{\GL_2}^{\mu^2}(\Gamma)$ is the collection of matrices $(A_1, \ldots, A_d)$ of the form
$$
	A_k = \begin{pmatrix} a_{11}^k & a_{12}^k \\ a_{21}^k & a_{22}^k\end{pmatrix}
$$
with $a_{11}^ka_{22}^k-a_{12}^ka_{21}^k \neq 0$ for all $k = 1, \ldots, d$ and the vectors $(a_{12}^1, a_{12}^2, \ldots, a_{12}^d), (a_{21}^1, a_{21}^2, \ldots, a_{21}^d) \in \CC^m$ do not vanish.

It is easier to study the complement $\mathcal{R}_{\GL_2}(\Gamma) - \mathcal{R}_{\GL_2}^{\mu^2}(\Gamma)$ of this space. This is the collection of the same tuples of matrices, but now satisfying that $(a_{12}^1, a_{12}^2, \ldots, a_{12}^d)$ or $(a_{21}^1, a_{21}^2, \ldots, a_{21}^d)$ are identically zero. In the case that $(a_{12}^1, a_{12}^2, \ldots, a_{12}^d) = (0, \ldots, 0)$, since $a_{11}^ka_{22}^k \neq 0$ for all $k$, we get a contribution of $(q-1)^{2d}q^d$ and analogously for $(a_{21}^1, a_{21}^2, \ldots, a_{21}^d) = (0, \ldots, 0)$. Substracting the common locus in which both vectors vanish, which counts as $(q-1)^{2d}$, and taking into account that $e(\mathcal{R}_{\GL_2}(\Gamma)) = e(\GL_2)^{d} = (q^2+q)^d(q-1)^{2d}$, we get that
$$
	e\left(\mathcal{R}_{\GL_2}^{\mu^2}(\Gamma)\right) = (q^2+q)^d(q-1)^{2d} - \left(2(q-1)^{2d}q^d - (q-1)^{2d}\right) = (q-1)^{2d}((q^2+q)^d - 2q^d +1).
$$

Therefore, the final count is 
\begin{equation}\label{eq:e-pol-11-2}
	e(\mathcal{X}_{(1,1)}(\Free{d})) = (q-1)^{2d} +  (q-1)^{2d-1}((q^2+q)^d - 2q^d +1).
\end{equation}

It is worth mentioning that this calculation can be done by hand in the case $d = 1$, i.e.\ $\Gamma = \ZZ$, and it agrees with the previous calculation. In this case, $\mathcal{R}_{\GL_2}(\ZZ) = \{(a, b, c, d) \in \CC^4 \,|\, ad-bc \neq 0\}$ and the action of $(\CC^*)^2$ is $(t_1, t_2) \cdot (a,b,c,d) \mapsto (a, t_1t_2^{-1}b, t_1^{-1}t_2c, d)$. Hence, the quotient space is given by setting the new variable $x = bc$ and thus $\mathcal{X}_{(1,1)}(\ZZ) = \{(a, x, d) \in \CC^3 \,|\, ad-x \neq 0\}$. Therefore, we get
$$
	\mathcal{X}_{(1,1)}(\ZZ) = \CC^3 - \{(a, x, d) \in \CC^3 \,|\, ad= x\} \cong \CC^3 - \CC^2,
$$
and thus $e(\mathcal{X}_{(1,1)}(\ZZ)) = q^3-q^2$. This agrees with (\ref{eq:e-pol-11-2}) for $d=1$.
\end{ex}

\begin{rem}
The previous calculation agrees with the one of $e(\cW_{H_2,\SL_2})$, as computed in Theorem \ref{thm:e-poly-sl2}. 
\end{rem}

\begin{ex}\label{ex:X11-Z2}
Extending the calculation of Example \ref{ex:totally-action-sl2}, we can even compute equivariant $E$-polynomials. Take $\Gamma = \Free{d}$, the free group in $d$ generators, and consider the action of $\ZZ_2$ on $\mathcal{R}_{\GL_2}(\Gamma)$ by simultaneous permutation of columns and rows, which descends to an action on $\mathcal{X}_{(1,1)}(\Gamma) = \mathcal{R}_{\GL_2}(\Gamma) / (\CC^*)^2$. Equivalently, this is the action by conjugation by the matrix $\scriptsize \begin{pmatrix}0 & 1 \\ 1 & 0\end{pmatrix}$.

First of all, recall that if we consider the action of $\ZZ_2$ on $X = (\CC^*)^2$ by permutation, then we have that $e(X)^+ = e(X/\ZZ_2) = q^2-q$, given by the trace and determinant of the diagonal matrix, and $e(X)^- = e(\CC^*)^2 - e(X/\ZZ_2) = 1-q$. Hence, the equivariant $E$-polynomial is given by
$$
	e_{\ZZ_2}(X) = e(X)^+\,T + e(X)^-\,N = (q^2-q)\,T + (1-q)\,N.
$$
In this way, we get that
\begin{align*}
	e_{\ZZ_2}((\CC^*)^{2d}) &= e_{\ZZ_2}(X)^{\otimes d} = \frac{1}{2}\left(e(X)^d + (e(X)^+ - e(X)^-)^d\right)\,T + \frac{1}{2}\left(e(X)^d - (e(X)^+ - e(X)^-)^d\right)\,N \\
	& = \frac{1}{2} \left((q-1)^{2d} + (q^2-1)^d\right)\,T + \frac{1}{2} \left((q-1)^{2d} - (q^2-1)^d\right)\,N.
\end{align*}

Now, let us come back to our calculation of the equivariant $\ZZ_2$-polynomial of $\mathcal{X}_{(1,1)}(\Gamma)$. This action is preserved under the decomposition $\mathcal{X}_{(1,1)}(\Gamma) = \mathcal{R}_{\GL_2}^{\mu^1}(\Gamma) / (\CC^*)^2 \sqcup \mathcal{R}_{\GL_2}^{\mu^2}(\Gamma) / \CC^*$. For the action on $\mathcal{R}_{\GL_2}^{\mu^1}(\Gamma) / (\CC^*)^2 = \mathcal{R}_{\GL_2}^{\mu^1}(\Gamma) = (\CC^*)^{2d}$, we directly get that $e_{\ZZ_2}(\mathcal{R}_{\GL_2}^{\mu^1}(\Gamma) / (\CC^*)^2) = e_{\ZZ_2}((\CC^*)^{2d})$ as computed above. For $\mathcal{R}_{\GL_2}^{\mu^2}(\Gamma)$, as above it is simpler to study its complement $Y = \mathcal{R}_{\GL_2}(\Gamma) - \mathcal{R}_{\GL_2}^{\mu^2}(\Gamma)$, which are the matrices with  $(a_{12}^1, a_{12}^2, \ldots, a_{12}^d)$ or $(a_{21}^1, a_{21}^2, \ldots, a_{21}^d)$ identically zero. In the subspace $Y_1 \subset Y$ where only one of these vectors is zero, there is a unique representative of the orbit which is upper triangular, and thus we get $e(Y_1)^+ = e(Y_1)^- = e(Y)/2 = (q-1)^{2d}(2q^d - 2)/2$ so
$$
	e_{\ZZ_2}(Y_1) = (q-1)^{2d}(q^{d}-1)\,T + (q-1)^{2d}(q^{d}-1)\,N. 
$$
The remaining part $Y_2$ is made of $d$ diagonal matrices and thus $e_{\ZZ_2}(Y_2) = e_{\ZZ_2}((\CC^*)^{2d})$. Taking into account that $e_{\ZZ_2}(\mathcal{R}_{\GL_2}(\Gamma)) = e(\mathcal{R}_{\GL_2}(\Gamma))\,T =  (q^2+q)^d(q-1)^{2d}\,T$, since $\ZZ_2$ acts trivially in cohomology because the $\ZZ_2$-action can be extended to a $\GL_2$-action, then we have
\begin{align*}
	e_{\ZZ_2}(\mathcal{R}_{\GL_2}^{\mu^2}(\Gamma)) & = e_{\ZZ_2}(\mathcal{R}_{\GL_2}(\Gamma)) - e_{\ZZ_2}(Y_1) - e_{\ZZ_2}(Y_2) \\
	&= (q-1)^{2d}((q^2+q)^d - q^d +1) \,T - (q-1)^{2d}(q^{d}-1)\,N - e_{\ZZ_2}((\CC^*)^{2d}).
\end{align*}

Now, observe that since the action of $\CC^*$ on $\mathcal{R}_{\GL_2}^{\mu^2}(\Gamma) $ is free, we have a $\ZZ_2$-equivariant $\CC^*$-principal bundle
$$
	\CC^* \to \mathcal{R}_{\GL_2}^{\mu^2}(\Gamma) \to \mathcal{R}_{\GL_2}^{\mu^2}(\Gamma) / \CC^*.
$$
Therefore, we have $e_{\ZZ_2}(\mathcal{R}_{\GL_2}^{\mu^2}(\Gamma)) = e_{\ZZ_2}(\mathcal{R}_{\GL_2}^{\mu^2}(\Gamma) / \CC^*)e_{\ZZ_2}(\CC^*)$. The action of $\ZZ_2$ on $\CC^*$ is given by $\lambda \mapsto \lambda^{-1}$ and thus $e_{\ZZ_2}(\CC^*) = qT - N$. Hence, using (\ref{eqn:pm-formula-fiber-forbase}), we get that
$$
	e_{\ZZ_2}(\mathcal{R}_{\GL_2}^{\mu^2}(\Gamma) / \CC^*) = \frac{q \, e(\mathcal{R}_{\GL_2}^{\mu^2}(\Gamma))^+ + e(\mathcal{R}_{\GL_2}^{\mu^2}(\Gamma))^-}{q^2-1}T + \frac{q\, e(\mathcal{R}_{\GL_2}^{\mu^2}(\Gamma))^- + e(\mathcal{R}_{\GL_2}^{\mu^2}(\Gamma))^+}{q^2-1}N.
$$

Putting all together, we finally find that
\begin{align*}
	e_{\ZZ_2}(\mathcal{X}_{(1,1)}(\Gamma)) = & \, e_{\ZZ_2}\left(\mathcal{R}_{\GL_2}^{\mu^1}(\Gamma) / (\CC^*)^2\right) +  e_{\ZZ_2}(\mathcal{R}_{\GL_2}^{\mu^2}(\Gamma) / \CC^*)  \\
	= & \, \left({q^{d+1}\left(q + 1\right)}^{d-1} {\left(q - 1\right)}^{2  d-1} - q^{d} {\left(q - 1\right)}^{2  d-1} + \frac{1}{2} \left(q{\left(q - 1\right)}^{2  d-1}  + {\left(q^{2} - q\right)} {\left(q^{2} - 1\right)}^{d-1}\right)\right)\,T \\
	& \, + \left( q^d{\left(q + 1\right)}^{d-1} {\left(q - 1\right)}^{2  d-1} -  q^{d}{{\left(q - 1\right)}^{2  d-1}} + \frac{1}{2} \left({q\left(q - 1\right)}^{2d-1} - {\left(q^{2} - q\right)} {\left(q^{2} - 1\right)}^{d-1}\right)\right)\,N.
\end{align*}
\end{ex}

\begin{ex} \label{ex:X111}
In the case $n = 3$ and $\Gamma = \Free{d}$, the possible partitions are $\mu^1 = \{\{1\}, \{2\}, \{3\}\}, \mu^2 = \{\{1,2\}, \{3\}\}, \mu^3 = \{\{1,3\}, \{2\}\}, \mu^4 = \{\{2, 3\}, \{1\}\}$ and $\mu^5 = \{\{1,2,3\}\}$, with $\mathcal{P}_0 = \{\mu^5\}$. Hence, we have
\begin{equation}\label{eq:X111}
 e(\mathcal{X}_{(1,1,1)}(\Gamma)) = \sum_{i=1}^4 e\left(\mathcal{R}_{\GL_3}^{\mu^i}(\Gamma) / (\CC^*)^3\right) +  \frac{1}{(q-1)^2} e(\mathcal{R}_{\GL_3}^{\mu^5}(\Gamma)).
\end{equation}
We count each stratum separately:
\begin{itemize}
	\item For $\mu^1 = \{\{1\}, \{2\}, \{3\}\}$, we have $\mathcal{R}_{\GL_3}^{\mu^1}(\Gamma) = (\CC^*)^{3d}$ with the trivial action, so it contributes with $(q-1)^{3d}$.
	\item For $\mu^2 = \{\{1,2\}, \{3\}\}$, we have $\mathcal{R}_{\GL_3}^{\mu^2}(\Gamma) / (\CC^*)^3 = {(\CC^*)^{d}} \times \mathcal{R}_{\GL_2}^{\{\{1,2\}\}}(\Gamma) / (\CC^*)^2$. By Example \ref{ex:totally-action-sl2}, this contributes as $(q-1)^{{{3d-1}\color{black}}}((q^2+q)^d - 2q^d +1)$. The cases $\mu^3 = \{\{1,3\}, \{2\}\}$ and $\mu^4 = \{\{2, 3\}, \{1\}\}$ are completely analogous.
	
	\item For $\mu^5 = \{\{1,2,3\}\}$, it is easier to compute its complement in $\mathcal{R}_{\GL_3}(\Gamma)$. It is given by the set of $d$ matrices $(a_{ij}^k)$ for $1 \leq i,j \leq 3$ and $1 \leq k \leq d$ such that at least one of the vectors $(a_{21}^k, a_{31}^k), (a_{12}^k, a_{13}^k), (a_{12}^k, a_{32}^k), (a_{21}^k, a_{23}^k), (a_{13}^k, a_{23}^k)$ or $(a_{31}^k, a_{32}^k)$ of $\CC^{2d}$ vanishes.
We will stratify into disjoint subsets as follows:
 \begin{itemize} 
 \item Exactly one of these vector vanishes, let us say $(a_{21}^k,a_{31}^k)=0$. This implies that, at the same time, the corresponding $2 \times 2$ minor
 has non-zero off-diagonal entries $(a_{23}^k)\neq 0, (a_{32}^k)\neq 0$ and the vector $(a_{12}^k, a_{13}^k) \neq 0$. Therefore, it has the form
 $$
(A_1, \ldots, A_d) = \left(\begin{pmatrix}a_{11}^1 & a_{12}^1 & a_{13}^1 \\ 0 & a_{22}^1 & a_{23}^1 \\ 0 & a_{32}^1 & a_{33}^1 \end{pmatrix}, 
\ldots, 				\begin{pmatrix}a_{11}^d & a_{12}^d & a_{13}^d \\ 0 & a_{22}^d & a_{23}^d \\ 0 & a_{32}^d & a_{33}^d \end{pmatrix}
\right).
$$
These correspond to non-splitting reducible representations that decompose into a one-dimensional representation and an irreducible two-dimensional one. There are $6$ disjoint copies of these as the row or column with zeros is well-determined. Therefore the $E$-polynomial is $6(q-1)^d (q^{2d}-1)e(S_d)$, where $S_d$ are the matrices in $\GL_2^d$ with
$(a_{23}^k)\neq 0, (a_{32}^k)\neq 0$.
Then
$e(S_d)= e(\GL_2)^d - (q-1)^{2d} (2q^d-1)$. All together, the contribution is
 \begin{align*}
 6(q-1)^d (q^{2d}-1)  & \left( q^d(q+1)^d(q-1)^{2d} - (q-1)^{2d} (2q^d-1) \right) \\
 &= 6(q-1)^{3d} (q^{2d}-1) ( q^{d}(q+1)^d-2q^d+1).
  \end{align*}
  
\item Exactly one horizontal and one vertical vector vanish, and are in the same location, let us say 
 $$
(A_1, \ldots, A_d) = \left(\begin{pmatrix}a_{11}^1 & 0 & 0 \\ 0 & a_{22}^1 & a_{23}^1 \\ 0 & a_{32}^1 & a_{33}^1 \end{pmatrix}, 
\ldots, 				\begin{pmatrix}a_{11}^d & 0 & 0 \\ 0 & a_{22}^d & a_{23}^d \\ 0 & a_{32}^d & a_{33}^d \end{pmatrix}
\right).
$$
As above, the off-diagonal entries of the minor cannot vanish so $(a_{23}^k)\neq 0, (a_{32}^k)\neq 0$. These correspond to reducible representations with a direct sum decomposition into a two-dimensional and a one-dimensional representation. There are $3$ copies of these. The $E$-polynomial is $3(q-1)^{3d} ( q^{d}(q+1)^d-2q^d+1)$.
 
\item Exactly one of the horizontal and one of the vertical vectors both vanish, and are in different locations, let us say 
 $$
(A_1, \ldots, A_d) = \left(\begin{pmatrix}a_{11}^1 & a_{12}^1 & a_{13}^1 \\ 0 & a_{22}^1 & a_{23}^1 \\ 0 & 0 & a_{33}^1 \end{pmatrix}, 
\ldots, 				\begin{pmatrix}a_{11}^d & a_{12}^d & a_{13}^d \\ 0 & a_{22}^d & a_{23}^d \\ 0 & 0 & a_{33}^d \end{pmatrix}
\right).
$$
where $(a_{12}^k)\neq 0$ and $(a_{23}^k)\neq 0$. These correspond to reducible representations with a full flag, but not splitting off a summand
(i.e.\ just one eigenvector). There are $6$ copies of these. The $E$-polynomial is $6 (q-1)^{3d} (q^d-1)^2 q^d$.

\item Exactly two vertical vectors vanish, but not the three, and also the corresponding case for two horizontal vectors, let us say
 $$
(A_1, \ldots, A_d) = \left(\begin{pmatrix}a_{11}^1 & 0 & a_{13}^1 \\ 0 & a_{22}^1 & a_{23}^1 \\ 0 & 0 & a_{33}^1 \end{pmatrix}, 
\ldots, 				\begin{pmatrix}a_{11}^d & 0 & a_{13}^d \\ 0 & a_{22}^d & a_{23}^d \\ 0 & 0 & a_{33}^d \end{pmatrix}
\right),
$$
with $(a_{13}^k)\neq 0, (a_{23}^k)\neq 0$. There are $6$ disjoint strata of these.
The $E$-polynomial is $ 6 (q-1)^{3d} (q^d-1)^2$.

\item Two vertical vectors vanish and two horizontal vectors vanish at the same time, but not the three. Let us say
 $$
(A_1, \ldots, A_d) = \left(\begin{pmatrix}a_{11}^1 & 0 & 0\\ 0 & a_{22}^1 & a_{23}^1 \\ 0 & 0 & a_{33}^1 \end{pmatrix}, 
\ldots, 				\begin{pmatrix}a_{11}^d & 0 &0\\ 0 & a_{22}^d & a_{23}^d \\ 0 & 0 & a_{33}^d \end{pmatrix}
\right),
$$
with $(a_{23}^k) \neq 0$. There are $6$ positions for the non-zero entry, hence we get $6 (q-1)^{3d}(q^d-1)$.

\item Diagonal matrices, which give $(q-1)^{3d}$.

\end{itemize}
Putting all together, we finally get that the complement of  $\mathcal{R}_{\GL_3}^{\mu^5}(\Gamma)$ has $E$-polynomial
$$	
(q-1)^{3d} \big( 3 (q + 1)^{d} (2 q^{3 d} - q^{d}) - 6 q^{3 d} + 6 q^{d} - 2 \big),
 $$
and thus
$$
	\mathcal{R}_{\GL_3}^{\mu^5}(\Gamma) = (q-1)^{3d}\left((q^2+q+1)^d(q+1)^dq^{3d}-3(q+1)^d(2q^{3d}-q^d)+6q^d(q^{2d}-1)+2\right).
$$
\end{itemize}

Therefore, plugging these calculations into (\ref{eq:X111}), we finally get
\begin{align*}
	e(\mathcal{X}_{(1,1,1)}(\Gamma)) & = (q-1)^{3d}+3(q-1)^{3d-1}((q^2+q)^d-2q^d+1) \\ 
	& \phantom{=} +(q-1)^{3d-2}\left((q^2+q+1)^d(q+1)^dq^{3d}-3(q+1)^d(2q^{3d}-q^d)+6q^d(q^{2d}-1)+2\right).
\end{align*}

\end{ex}

%%%%%%%%%%%%%%%%%%%%%%%%%%%%%%%%%%%%%%%%%%%%%%
\section{$\SL_3(\CC)$-character varieties of torus links} \label{sec:sl3charactervariety}
%%%%%%%%%%%%%%%%%%%%%%%%%%%%%%%%%%%%%%%%%%%%%%%%%%%%%

We turn now to the case $G=\SL_3(\CC)$, shortened as $\SL_3$ to ease notation any time that is needed. Let us write
\begin{align}
H_1 & := \bm{\mu}_3 \Id =\lbrace \xi \Id |\,  \xi \in \bm{\mu}_3 \rbrace, \nonumber \\ 
H_2 & := \left\lbrace \begin{pmatrix} \lambda \Id_2 & 0 \\ 0 & \lambda^{-2} \end{pmatrix} | \, \lambda \in \mathbb{C}^{\ast} \right\rbrace \cong \CC^{\ast}, \label{def:subgroupsSL3}\\
H_3 & := \left\lbrace \diag(\lambda_1,\lambda_2,\lambda_3) |\, \lambda_1 \lambda_2 \lambda_3 =1 \right\rbrace \cong (\CC^{\ast})^2, \nonumber \\
H_4 & := \left\lbrace \begin{pmatrix} P & 0 \\ 0 & \det(P)^{-1} \end{pmatrix}  |\,  P\in \GL_2(\CC) \right\rbrace \cong \GL_2(\CC) \nonumber,
\end{align}
which are the stabilizers by conjugation in $\SL_3(\CC)$ of $\GL_3(\CC)$, $\GL_2(\CC)\times \GL_1(\CC)$, $\GL_1(\CC)\times \GL_1(\CC)\times \GL_1(\CC)$ and $\bm{\mu}_2\Id_2\times \GL_1(\CC)$, respectively. 
Besides, computing their normalizers, we note that the only relevant combinations with non-connected symmetry group are $N_{H_3, \SL_3} = N_{H_3, H_3} \cong S_3$ and $N_{H_3, H_4} = \ZZ_2$.

When $d=1$, the character varieties of torus knot groups were described in \cite{munozporti:2016}. We refine here that description in terms of stabilizers, indicating that out of the 15 possible strata $V_{H_i,H_j}$, with $H_i < H_j$, irreducibility reduces the total count to 9 non-empty strata. The character variety $\X(n,m)$ can be stratified as:

\begin{itemize}
\item $\X(n,m)^{\ast}$, the stratum of irreducible representations, which is made of:
\begin{itemize}
\item[*] $\frac{1}{12}(m-1)(m-2)(n-1)(n-2)$ components of dimension $4$.
\item[*] $\frac{1}{2}(n-1)(m-1)(n+m-4)$ components isomorphic to $(\CC^{\ast})^2-\lbrace x+y=1 \rbrace$.
\end{itemize}
Irreducibility forces again that the stabilizer of any of these representations is a multiple of the identity, and so is $A^n=B^m$, hence $\Stab(A^n)=\SL_3$. We get $\X(n,m)^{\ast}=\cV_{H_1,\SL_3}$.

\item $[\frac{n-1}{2}][\frac{m-1}{2}]$ components of partially reducible representations into one representation of rank $2$ and another of rank $1$. Each of these components is isomorphic to $(\CC-\lbrace 0,1 \rbrace) \times \CC^{\ast}$. With respect to a certain basis, each representation of this stratum can be written as
$$
A = \begin{pmatrix} A' & 0 \\ 0 & \det(A')^{-1} \end{pmatrix}, \qquad B= \begin{pmatrix} B' & 0 \\ 0 & \det(B')^{-1} \end{pmatrix},
$$
where $(A',B')$ is an irreducible $\GL_2(\CC)$-representation of the torus knot, hence the stabilizer is  $H_2$.  Let us assume first that $n,m$ are both odd.

There are two cases:
\begin{enumerate}

\item $\cV_{H_2, \SL_3}$, when $A^n=B^m\in H_1$, hence $(A')^n=(B')^m\in \bm{\mu}_3\Id_2$, and  the determinant representation $(\det A', \det B') \in \X_{\GL_1}(n,m)$ is parametrized by a unique $ t\in \bm{\mu}_{3mn}$.  From the description in \cite{munozporti:2016}, equivalence classes of irreducible $\GL_2(\CC)$-representations can be described as $ \left( \X_{\SL_2}(n,m)^{\ast}\times \CC^{\ast} \right) / \bm{\mu}_2$, where the $\CC^{\ast}$-factor, that corresponds to the determinant representation, takes into account the $\CC^{\ast}$-action $\gamma (A',B')=(\gamma^m A',\gamma^n B')$ on $\X_{\GL_2}(n,m)$, with kernel $\bm{\mu}_2$. In this case, since $(\det A', \det B')\in \bm{\mu}_{3mn}$, we obtain
\begin{equation}
\cV_{H_2,\SL_3} \cong \left( \X_{\SL_2}(n,m)^{\ast}\times \bm{\mu}_{3mn} \right) / \bm{\mu}_2 \label{eqn:VH2SL3}
\end{equation}
The $\bm{\mu}_2$-action on $\X_{\SL_2}(n,m)$ interchanges components of $\X_{\SL_2}(n,m)^{\ast}$.

\item $\cV_{H_2, H_4}$, if $A^n=B^m \not \in H_1$. Since the $\GL_2(\CC)$-representation is irreducible, $(A')^n=(B')^m$ is a multiple of the identity different from $\bm{\mu}_{3mn}$, so the stabilizer of $A^n=B^m$ is $H_4$. The same description yields
\begin{equation} \label{eqn:VH2H4}
\cV_{H_2,H_4} \cong \left( \X_{\SL_2}(n,m)^{\ast}\times (\CC^{\ast}-\bm{\mu}_{3mn}) \right) / \bm{\mu}_2\, .
\end{equation}

\end{enumerate}

When $n$ is even, $\bm{\mu}_2 \subset \bm{\mu}_{3mn}$, so there are $(m-1)/2$ components of partially reducible representations in $\X_{\SL_2}(n,m)^{\ast}$ that remain fixed by $\bm{\mu}_2$, hence they also show up in $\cV_{H_2,\SL_3}$ and $\cV_{H_2,H_4}$. They are isomorphic to $\lbrace (u,v)\in \CC^2 |\, v\neq 0, v\neq u^2 \rbrace$, c.f.\ \cite{munozporti:2016}.

\item $\X(n,m)^{\textrm{TR}}$, the stratum of totally reducible representations, i.e.\ representations that decompose into three $1$-dimensional (irreducible) representations. It is isomorphic to $\CC^2$ via the coefficients of the characteristic polynomial of a common root of $A$ and $B$. More precisely, given $(A,B)\in \X(n,m)^{\textrm{TR}}$, there exists $t_i\in \CC^{\ast}, i=1,2,3,$ such that
$$A=\diag(t_1^m,t_2^m,t_3^m),\: B=\diag(t_1^n,t_2^n,t_3^n),$$
where $t_1t_2t_3=1$. The space of such triples $(t_1, t_2, t_3)$, without the determinant condition, is $\Sym^3(\CC^{\ast})$, which is isomorphic to $\CC^2\times \CC^{\ast}$ via the coefficients of the characteristic polynomial. But here the determinant condition forces the last factor to be equal to $1$, hence we get $\CC^2$.

Besides, total reducibility implies that we may assume that $(A,B)$ are simultaneously diagonal with respect to a certain basis. We distinguish the following subcases:
\begin{enumerate}
\item If $A,B \in H_1$, all stabilizers are equal to $\SL_3$. It corresponds to the case $\cV_{\SL_3,\SL_3}=\tilde{\cV}_{\SL_3,\SL_3}$, which consists of three points inside $\X(n,m)^{{\textrm{TR}}}\cong \CC^2$.
\item If $A,B \in H_2$ but $A \not\in H_1$ or $B \not\in H_1$, then both $A$ and $B$ are simultaneously diagonalizable and either $A$ or $B$ have a repeated eigenvalue which does not belong to $\bm{\mu}_3$. This corresponds to the cases $\lbrace t_i=t_j\neq t_k \rbrace$ inside $\X(n,m)^{{\textrm{TR}}}$, which are equivalent by the $S_3$-action, so that we may fix $t_1=t_2\neq t_3$. The stabilizer of these representations $(A,B)$ is thus $H_4$. There are two possibilities:
\begin{itemize}
\item[$\star$] $A^n=B^m\in H_1$, which corresponds to $\cV_{H_4,\SL_3}$. In this case, the repeated eigenvalue satisfies $\lambda^n=\mu^m=\xi \in \bm{\mu}_3$, so that $(\lambda,\mu)\in \bm{\mu}_{3n}\times\bm{\mu}_{3m} - \Delta$, where $\Delta=\lbrace (\xi,\xi)|\, \xi \in \bm{\mu}_3 \rbrace \subset \bm{\mu}_{3n}\times\bm{\mu}_{3m}.$ There is a unique  $t\in \CC^{\ast}$ that satisfies that $t^m=\lambda, t^n=\mu$, so that $t\in \bm{\mu}_{3mn}- \bm{\mu}_3$ and the component consists of $3mn-3$ points.
\item[$\star$] $A^n=B^m \in H_2-H_1$, that is, $\cV_{H_4, H_4}$. The same discussion yields that now $t \in \CC^{\ast} - \bm{\mu}_{3mn}$. The stratum is isomorphic to $\CC^{\ast}-\bm{\mu}_{3mn}$.
\end{itemize} 
\item If $A,B \in H_3$ but $A \not\in H_2$ or $B \not\in H_2$. In this case, either $A$ or $B$ has three different eigenvalues, so that this case corresponds to $t_1\neq t_2 \neq t_3 \in \Sym^3(\CC^{\ast})$. In this case, the stabilizer of $(A, B)$ is $H_3$ and the action of $N_{H_3} = S_3$ is free. Again, we get three possibilities depending on where $A^n=B^m$ belongs:

\begin{itemize}
\item[$\star$] $A^n=B^m\in H_1$, which corresponds to $\cV_{H_3,\SL_3}$. These points include those lying in the closure of $\X(n,m)^{\ast}$ in $\X(n,m)^{\mathrm{TR}}$. Since $t_1^{mn}=t_2^{mn}=t_3^{mn}\in \bm{\mu}_3$, we need to count points belonging to (here $t_2=t_1\epsilon_1$, $t_3=t_1\epsilon_2$)
$$
\hspace*{3cm}
V^0_{H_3,\SL_3}  \cong \lbrace (t_1,\epsilon_1,\epsilon_2) \in \CC^{\ast}\times (\bm{\mu}_{mn})^2 \mid t_1^3\epsilon_1\epsilon_2=1, \epsilon_1\neq 1, \epsilon_2 \neq 1, \epsilon_1 \neq \epsilon_2 \rbrace.
$$

Using the inclusion-exclusion principle, we get a total of $3m^2n^2-9mn+6$ points in $V^0_{H_3,\SL_3}$, which get identified under the $S_3$-action, 
leaving $\frac12 (m^2n^2-3mn+2)$ points in $\cV_{H_3,\SL_3} = V^0_{H_3,\SL_3}  / N_{H_3}$  (see \cite{gonmun:2022}).

\item[$\star$] $A^n=B^m \in H_2-H_1$, that is, $\cV_{H_3, H_4}$. There exists $i,j\in \lbrace 1,2,3 \rbrace$ such that $t_i\neq t_j$, but $t_i^{mn}=t_j^{mn}=\lambda \in \CC^{\ast}$, so that $t_i=\epsilon t_j$, $\epsilon \in \bm{\mu}_{mn}^{\ast}$. The space can be described as
$$
\hspace*{3.4cm}
V^0_{H_3,H_4}=\lbrace (t_1,t_2,\epsilon)\in (\CC^{\ast})^2 \times \bm{\mu}_{mn} \mid  t_1^2 t_2 \epsilon = 1, t_2 \neq t_1, t_2\neq t_1 \varepsilon \:\: \forall \varepsilon \in \bm{\mu}^{\ast}_{mn} \rbrace.
$$
The $S_3$-action reduces to a $\ZZ_2$-action that takes $(t_1,t_2,\epsilon)\mapsto (t_1\epsilon,t_2,\epsilon^{-1})$. It is a collection of punctured lines that get identified depending on the value of $\epsilon$. Its equivariant polynomial was computed in \cite{gonmun:2022} and it is given by
$$
e(V^0_{H_3,H_4}) = \left( \left\lfloor \frac{mn}{2} \right\rfloor (q-1)- \frac{3mn(mn-1)}{2} \right) T + \left( \left\lfloor \frac{mn-1}{2} \right\rfloor (q-1)- \frac{3mn(mn-1)}{2} \right) N.
$$

\item[$\star$] $A^n=B^m \in H_3-H_2$, that is, $\cV_{H_3,H_3}$. Either $A$ or $B$ has three different eigenvalues and so does $A^n$. We get all the remaining cases in $\Sym^3(\CC^{\ast})$ that are left when we exclude $\cV^0_{H_3,\SL_3}$ and $\cV^0_{H_3,H_4}$, so $t_1\neq t_2 \neq t_3$. That is,
$$
\hspace*{2.4cm}
V^0_{H_3,H_3}=\lbrace \diag(t_1,t_2,t_3) \mid t_1\neq t_2\neq t_3, t_1t_2t_3=1, t_i^{mn}\neq t_j^{mn}, i\neq j \rbrace.
$$
\end{itemize}
\end{enumerate}
\end{itemize}

This description of $\X(n,m)$ for $\SL_3(\CC)$ leads to a stratification of $\X^d(n,m)$, due to Section \ref{sec:representationstrata} and Theorem \ref{eq:iso-w}, from where we recall the isomorphism
\begin{equation}
	\mathcal{W}_{H_1, H_2} := W_{H_1, H_2} \sslash G \cong \left(V_{H_1, H_2}^0 \times H_2^{d-1} \sslash H_1 \right) \sslash N_{H_1}.
\end{equation}

We get nine corresponding strata $\cW_{H_i,H_j}$ for $\X^d(n,m)$, which we enumerate:
\begin{itemize}
\item $e(\cW_{H_1,\SL_3})=e(\pi^{-1}(\X^{\ast}(n,m)))=e(\X(n,m)^*) e(\sltres)^{d-1}$, due to Corollary \ref{cor:irreduciblelocus}.
\item $e(\cW_{H_2,\SL_3})=e\left(\left( \X_{\SL_2}(n,m)^{\ast}\times \bm{\mu}_{3mn} \right) / \bm{\mu}_2\right) e\left( \sltres^{d-1} \sslash H_2 \right)$, by the description of $V^0_{H_2,\SL_3}$ from \eqref{eqn:VH2SL3}. 
\item $e(\cW_{H_2,H_4}) = e\left(\left( \X_{\SL_2}(n,m)^{\ast}\times (\CC^{\ast}-\bm{\mu}_{3mn}) \right) / \bm{\mu}_2\right) e\left( H_4^{d-1}\sslash H_2 \right)$, given the description of $V^0_{H_2,H_4}$ from \eqref{eqn:VH2H4}.

\item $e(\cW_{\SL_3,\SL_3})=e(V^0_{\SL_3,\SL_3}) e\left( \sltres^{d-1} \sslash \sltres \right)$. 
Since $V^0_{\SL_3,\SL_3}$ consists of three points, we obtain three copies of the $\sltres$-character variety of the free group in $d-1$ generators.

\item 
$e(\cW_{H_4,\SL_3}) = e(V^0_{H_4,\SL_3}) e\left( \sltres^{d-1} \sslash H_4 \right)$. The space is a collection of $3mn-3$ copies of $\sltres^{d-1} \sslash H_4$.

\item $e(\cW_{H_4,H_4})=e(V^0_{H_4,H_4}) e(H_4^{d-1} \sslash H_4) =e\left( \CC^{\ast}-\bm{\mu}_{3mn} \right) e(H_4^{d-1} \sslash H_4)$.

\item $\cW_{H_3,\SL_3} = \widetilde{W}{}_{H_3,\SL_3}/S_3 \cong (V^0_{H_3,\SL_3} \times \sltres^{d-1}\sslash H_3)/S_3$. Since $V^0_{H_3,\SL_3}$ is a collection of points, the space is a discrete set of components isomorphic to $\sltres^{d-1} \sslash H_3$ that get identified under the $S_3$-action.

\item $\cW_{H_3,H_4} = (V^0_{H_3,H_4}\times H_4^{d-1}\sslash H_3)/\ZZ_2$. In this case, $V^0_{H_3,H_4}$ is a collection of punctured copies of $\CC^{\ast}$, with a residual action of $\ZZ_2$.

\item $\cW_{H_3,H_3}$. Now $A$ or $B$ are simultaneously diagonalizable and one of them has three different eigenvalues, and this is also the case for $A^n$. Again, $N_{H_3}=S_3$, so we get that
$$
\cW_{H_3,H_3}=(V^0_{H_3,H_3}\times H_3^{d-1})/ S_3,
$$
taking into account that the action by conjugation of $H_3$ on $H_3$ is trivial, thus $H_3^{d-1}\sslash H_3 \cong H_3^{d-1}$.     

\end{itemize}

%%%%%%%%%%%%%%%%%%%%%%%%%%%%%%%%%%%%%%%%%%%
\section{GIT quotients of $\SL_3(\CC)$}\label{sec:git-quotients}
%%%%%%%%%%%%%%%%%%%%%%%%%%%%%%%%%%%%%%%%%%%

We begin this section by computing the $E$-polynomials of some quotients of $\sltres^{d-1}$ by the subgroups $H_i$ that were defined in \eqref{def:subgroupsSL3}. In all cases, the action of $P\in H_i$ is given by simultaneous conjugation, taking $(A_1,\ldots,A_{d-1}) \mapsto (PA_1P^{-1},\ldots,PA_{d-1}P^{-1})$.

\begin{prop} 
We have
\begin{equation*} \label{eqn:epolySL3H2}
e(\sltres^{d-1}\sslash H_2) = (q^3-q)^{d-1}\left(q^{3d-3}(q+1)^{d-1}(q-1)^{d-2} -(2q^{2d-2}-1)(q-1)^{d-2}+(q-1)^{d-1} \right).
\end{equation*}
\end{prop}
\begin{proof}
Explicitly, the action given by $\diag(\lambda,\lambda,\lambda^{-2})\in H_2$ on $A_i=(a^i_{jl})$ takes
$$
\begin{pmatrix} a_{11}^i & a_{12}^i & a_{13}^i \\ a_{21}^i & a_{22}^i & a_{23}^i \\ a_{31}^i & a_{32}^i & a_{33}^i \end{pmatrix} \longrightarrow \begin{pmatrix} a_{11}^i & a_{12}^i & \lambda^{3} a_{13}^i \\ a_{21}^i & a_{22}^i & \lambda^3 a_{23}^i \\ \lambda^{-3} a_{31}^i & \lambda^{-3} a_{32}^i & a_{33}^i \end{pmatrix}.
$$ 
First of all, note that the action is trivial if $A_i\in H_4$, that is, if $a^i_{31}=a^i_{32}=a^i_{13}=a^i_{23}=0$ for all $i=1,\ldots,d-1$. That is,
$$
e(H_4^{d-1}/H_2)=e(H_4^{d-1})= e(\GL_2(\CC))^{d-1}= (q-1)^{d-1}(q^3-q)^{d-1}.
$$
The cases $W_1=\lbrace (a^i_{31})=(a^i_{32})=0\in \CC^{d-1} \rbrace$ and $W_2=\lbrace (a^i_{13})=(a^i_{23})=0\in \CC^{d-1} \rbrace$ yield $S$-equivalent representations to those belonging to $H_4^{d-1}$, where $W_1$ and $W_2$ intersect. Their $E$-polynomial is
$$
e(W_1)=e(W_2)=e(\CC^{2d-2})e(\GL_2(\CC)^{d-1})=q^{2d-2}(q-1)^{d-1}(q^3-q)^{d-1}.
$$

All the remaining representations $(A_i) \in \sltres^{d-1} - (W_1\cup W_2)$ are stable points for the $\CC^{\ast}$-action, so the action is free. As a consequence:
\begin{align*}
e(\sltres^{d-1} \sslash H_2)   & = e\left( (\sltres^{d-1}-(W_1\cup W_2))/H_2 \right) + e(H_4^{d-1}/H_2) \\
 & = \frac{(q^3-1)^{d-1}q^{2d-2}(q^3-q)^{d-1}-(2q^{2d-2}-1)(q-1)^{d-1}(q^3-q)^{d-1}}{q-1} + (q-1)^{d-1}(q^3-q)^{d-1} \\
& =(q^3-q)^{d-1}\left(q^{3d-3}(q+1)^{d-1}(q-1)^{d-2} -(2q^{2d-2}-1)(q-1)^{d-2}+(q-1)^{d-1} \right).
\end{align*}
\end{proof}

\begin{prop}\label{prop:SL3-h4}
We have
\begin{align*}
e(\sltres^{d-1}\sslash H_4) & = \frac{1}{2}\left( (q-1)^{2d-2} + (q^2-1)^{d-1}\right) + (q-1)^{2d-3}((q^2+q)^{d-1} - 2q^{d-1} +1) \\
	& \quad + (q - 1)^{d-1}\big((q - 1)^{d-2}q^{d-2}((q + 1)^{d-2} - 1) + \frac12 (q - 1)^{d-2} - \frac12 (q + 1)^{d-2} \big)   \\ &  \quad + (q-1)^{2d-4} \left( (q^2+q+1)^{d-1}(q+1)^{d-2}q^{3d-4}-(q^2+q)^{d-2}(2q^{2d-2}-1) \right. \\ & \quad  \left. - q^{d-2}(q^{d-1} - 1)^2 - (2q^{2d-3}-1) ((q^2+q)^{d-1} - 2q^{d-1} + 1)  \right).
\end{align*}
\end{prop}

\begin{proof}
First of all, let us describe explicitly the action. Given $Q=\left(\begin{array}{c | c} P & 0 \\ \hline 0 & \det P^{-1} \end{array}\right) \in H_4$, where $P \in \GL_2(\CC)$, its action by conjugation on $\rho = (A_1, \ldots, A_{d-1})$ takes 
$$
Q (A_k) Q^{-1}=\left( \begin{array}{c | c}  B_k & \left(\begin{matrix} a^k_{13} \\ a^k_{23} \end{matrix} \right) \\ \hline  \left( \begin{matrix} a_{31}^k & a_{32}^k \end{matrix} \right) & a_{33}^k \end{array} \right) = \left( \begin{array}{c | c}  PB_k  P^{-1} & \det(P)P \begin{pmatrix} a_{13}^k \\ a_{23}^k \end{pmatrix} \\ \hline  \det(P)^{-1}\begin{pmatrix} a_{31}^k & a_{32}^k \end{pmatrix} P^{-1} & a_{33}^k \end{array} \right),
$$
where $k=1,\ldots, d-1$.

Consider the splitting $\CC^3_{\bullet} = \langle e_1, e_2 \rangle \oplus \langle e_3 \rangle$. Our space is precisely the $\lambda$-character variety $\mathcal{R}_{\SL(\CC^3_{\bullet})}(\Free{{d-1}}) \sslash H_4$. By Proposition \ref{prop-polystable-Vlambda}, the polystable locus coincides with the $\CC^3_{\bullet}$-semisimple representations. Let us decompose $\mathcal{R}_{\SL(\CC^3_{\bullet})}(\Free{{d-1}})$ according to its semisimple type.
\begin{itemize}
	\item If $\langle e_3 \rangle$ is an invariant subspace. In that case, the semisimple representations have the form
	$$
		\rho = (A_k) = \left(\left(\begin{array}{c | c } B_k & 0 \\ \hline 0 & \det B_k^{-1}\end{array}\right)\right),
	$$
	with $(B_1, \ldots, B_{d-1})$ forming an element of $\mathcal{R}_{\GL_2}(\Free{{d-1}})$. The action of $H_4$ on these matrices coincides exactly with the action of $\GL_2$ on $\mathcal{R}_{\GL_2}(\Free{{d-1}})$, so we get that the quotient space agrees with the $\GL_2$-character variety of the free group $\mathcal{R}_{\GL_2}(\Free{{d-1}}) \sslash \GL_2$. The reducible part of this quotient  is $(\CC^*)^{2d-2}/\ZZ_2$ which, by \cite{gon:2023}, has $E$-polynomial equal to
$$
	e\left((\CC^*)^{2d-2}/\ZZ_2\right) = \frac{1}{2}\left((q-1)^{2d-2} + (q^2-1)^{d-1}\right).
$$
On the other hand, the irreducible part of this quotient was computed in \cite[Lemma 6.5]{florentinonozadzamora:2021}, so adding both contributions we get
\begin{align*}
	e(\mathcal{R}_{\SL(\CC^3_{\bullet})}^{1}(\Free{{d-1}}) \sslash H_4) & = e(\mathcal{R}_{\GL_2}(\Free{{d-1}}) \sslash \GL_2) =  \frac{1}{2}\left( (q-1)^{2d-2} + (q^2-1)^{d-1}\right) \\ & \phantom{=}	+ (q - 1)^{d-1}\big((q - 1)^{d-2}q^{d-2}((q + 1)^{d-2} - 1) + \frac12 (q - 1)^{d-2} - \frac12 (q + 1)^{d-2} \big).
\end{align*}

	Furthermore, with this information at hand, we can also compute the reducible representations, not necessarily semisimple, for which $\langle e_3\rangle$ is an invariant subspace or $\langle e_1, e_2\rangle$ is an invariant plane. This corresponds to representations of the form
	$$
		(A_k) = \left(\left(\begin{array}{c | c } B_k & v_k \\ \hline w_k & \det B_k^{-1}\end{array}\right)\right),
	$$
with $v_k, w_k \in \CC^2$ and such that $(v_1, \ldots, v_{d-1})$ or $(w_1, \ldots, w_{d-1})$ vanish. The upper-left corner has $E$-polynomial equal to $e(\mathcal{R}_{\GL_2}(\Free{{d-1}})) = (q-1)^{2d-2}q^{d-1}(q+1)^{d-1}$, so this stratum $\mathcal{R}_{\SL(\CC^3_{\bullet})}^{1}(\Free{{d-1}})$ of reducible representations has $E$-polynomial
$$
	e(\mathcal{R}_{\SL(\CC^3_{\bullet})}^{1}(\Free{{d-1}})) = (q-1)^{2d-2}q^{d-1}(q+1)^{d-1} \left(2q^{2d-2} - 1\right).
$$

	\item If the representation is reducible and semisimple, but neither $\langle e_1, e_2 \rangle$ nor $\langle e_3\rangle$ are invariant subspaces, then there must exist a $1$-dimensional invariant subspace $V \subset \langle e_1, e_2\rangle$. After conjugation, we can suppose that $V = \langle e_1\rangle$ is the invariant subspace and thus the representation has the form
	$$
		(A_k) = \left(\left(\begin{array}{c c | c } \Delta_k & 0 & 0  \\ 0 & a_k & b_k \\ \hline 0 & c_k & d_k \end{array}\right)\right),
	$$
	with $\Delta_k = (a_kd_k - b_kc_k)^{-1}$ and the representation
$$
	\left(\begin{pmatrix}a_1 & b_1 \\ c_1 & d_1\end{pmatrix}, \ldots, \begin{pmatrix}a_{d-1} & b_{d-1} \\ c_{d-1} & d_{d-1}\end{pmatrix}\right) \in \mathcal{R}_{\GL_2}(\Free{{d-1}})
$$
being irreducible. The residual action on the bottom-right corner is the action of diagonal matrices $\CC^* \times \CC^*$ by conjugation. Hence, the quotient of this stratum $\mathcal{R}_{\SL(\CC^3_{\bullet})}^{2}(\Free{{d-1}})$ of representations is isomorphic to $\mathcal{R}_{\GL_2}^{\mu^2}(\Free{{d-1}}) / \CC^*$ as given in Example \ref{ex:totally-action-sl2} and thus its $E$-polynomial is
$$
	e(\mathcal{R}_{\SL(\CC^3_{\bullet})}^{2}(\Free{{d-1}}) \sslash H_4) = e(\mathcal{R}_{\GL_2}^{\mu^2}(\Free{{d-1}}) / \CC^*) = (q-1)^{2d-3}((q^2+q)^{d-1} - 2q^{d-1} +1).
$$

Let us now count the representations of this form, before taking the quotient. First, suppose that they have an invariant $1$-dimensional subspace $\ell  \subset \langle e_1, e_2 \rangle$, but the spaces $\langle e_1, e_2\rangle$ and $\langle e_3 \rangle $ are not invariant. Let us denote this stratum by $\mathcal{R}_{\SL(\CC^3_{\bullet})}^{2, \alpha}(\Free{{d-1}})$. Since the invariant subspace $\ell_\rho \subset \langle e_1, e_2 \rangle$ of a representation $\rho$ is uniquely determined, we have a natural fibration
$$
	\mathcal{R}_{\SL(\CC^3_{\bullet})}^{2, \alpha}(\Free{{d-1}}) \to \mathbb{P}^1, \qquad \rho \mapsto \ell_\rho.
$$
Obviously, this fibration has trivial monodromy since $\mathbb{P}^1$ is simply-connected. Hence, we have $e(\mathcal{R}_{\SL(\CC^3_{\bullet})}^{2, \alpha}(\Free{{d-1}})) = e(\mathcal{R}_{0}^{2, \alpha}(\Free{{d-1}}))(q+1)$, where $\mathcal{R}_{0}^{2, \alpha}(\Free{{d-1}})$ is the fiber of the previous fibration.

Over $\ell_\rho = \langle e_1\rangle$, the fiber $F$ is given explicitly by representations of the form
$$
		\rho = (A_k) = \left(\left(\begin{array}{c c | c } \Delta_k & \alpha_k & \beta_k  \\ 0 & a_k & b_k \\ \hline 0 & c_k & d_k \end{array}\right)\right),
$$ 
with $\Delta_k = (a_kd_k - b_kc_k)^{-1}$. Notice that, in this case, the representation
$$
	\rho_0=\left(\begin{pmatrix}a_1 & b_1 \\ c_1 & d_1\end{pmatrix}, \ldots, \begin{pmatrix}a_{d-1} & b_{d-1} \\ c_{d-1} & d_{d-1}\end{pmatrix}\right) \in \mathcal{R}_{\GL_2}(\Free{{d-1}})
$$
might not be irreducible. We also have that $(c_1, \ldots, c_{d-1}) \neq (0, \ldots, 0)$ (otherwise, $\langle e_1, e_2 \rangle$ would be an invariant subspace) and $(\beta_1, b_1, \ldots, \beta_{d-1}, b_{d-1}) \neq (0, \ldots, 0)$ (otherwise, $\langle e_3\rangle$ would be an invariant subspace). If we consider the map
\begin{equation}\label{eq:fibration-2-1}
	F \to \mathcal{R}_{\GL_2}(\Free{{d-1}}), \qquad \rho \mapsto \rho_0,
\end{equation}
its fibers are the possible choices of $(\alpha_1, \beta_1, \ldots, \alpha_{d-1}, \beta_{d-1}) \in \CC^{2d-2}$ fulfilling the constraints to belong to this stratum, given a fixed representation $\rho_0$. Notice that, in $\rho_0$, we must have that $(c_1, \ldots, c_{d-1}) \neq (0, \ldots, 0)$, so the $E$-polynomial of the image of (\ref{eq:fibration-2-1}) is
\begin{align*}
	e&(\mathcal{R}_{\GL_2}(\Free{{d-1}}) \cap \{(c_k) \neq 0\}) = e(\mathcal{R}_{\GL_2}(\Free{{d-1}})) -  e(\mathcal{R}_{\GL_2}(\Free{{d-1}}) \cap \{c_k = 0\})  \\
	&= (q-1)^{2d-2}q^{d-1}(q+1)^{d-1} - (q-1)^{2d-2}q^{d-1} =  (q-1)^{2d-2}q^{d-1}((q+1)^{d-1} - 1).
\end{align*}
With this in mind, we have two options:
\begin{itemize}
	\item If $b_k = 0$ for all $k = 1, \ldots, d-1$, then the fiber is given by vectors $(\alpha_1, \beta_1, \ldots, \alpha_{d-1}, \beta_{d-1})$ with $(\beta_1, \ldots, \beta_{d-1}) \neq (0, \ldots, 0)$. This implies that the fiber is isomorphic to $\CC^{d-1} \times (\CC^{d-1} - \{0\})$. The possible representations $\rho_0$ of the free group with $b_k = 0$ for all $k$ and $(c_1, \ldots, c_{d-1}) \neq (0, \ldots, 0)$ have $E$-polynomial
	\begin{align*}
	e&(\mathcal{R}_{\GL_2}(\Free{{d-1}}) \cap \{b_k = 0\}) -  e(\mathcal{R}_{\GL_2}(\Free{{d-1}}) \cap \{b_k, c_k = 0\}) \\
	& = (q-1)^{2d-2}q^{d-1} - (q-1)^{2d-2} =  (q-1)^{2d-2}(q^{d-1} - 1).
\end{align*}
Hence, this stratum contributes with $E$-polynomial
$$
q^{d-1}(q-1)^{2d-2}(q^{d-1} - 1)^2.
$$
	\item If $b_k \neq 0$ for some $k$, then the fiber is given by vectors $(\alpha_1, \beta_1, \ldots, \alpha_{d-1}, \beta_{d-1}) \in \CC^{2d-2}$. This implies that the fiber is isomorphic to $\CC^{2d-2}$. The possible representations $\rho_0$ of the free group with $(b_1, \ldots, b_{d-1}) \neq (0, \ldots, 0)$ and $(c_1, \ldots, c_{d-1}) \neq (0, \ldots, 0)$ have $E$-polynomial
	\begin{align*}
	e&(\mathcal{R}_{\GL_2}(\Free{{d-1}}) \cap \{(c_k) \neq 0\}) - e(\mathcal{R}_{\GL_2}(\Free{{d-1}}) \cap \{(c_k) \neq 0, (b_k) = 0\}) \\
	& = (q-1)^{2d-2} (q^{d-1}(q+1)^{d-1} - 2q^{d-1} + 1),
\end{align*}
and thus it contributes with
$$
q^{2d-2}(q-1)^{2d-2}(q^{d-1}(q+1)^{d-1} - 2q^{d-1} + 1).
$$
\end{itemize}
Therefore, adding up all the contributions, we get
\begin{align*}
e&(\mathcal{R}_{\SL(\CC^3_{\bullet})}^{2, \alpha}(\Free{{d-1}})) =(q-1)^{2d-2}(q+1)  \left(q^{d-1}(q^{d-1} - 1)^2 + q^{2d-2} (q^{d-1}(q+1)^{d-1} - 2q^{d-1} + 1)\right).
\end{align*}

Analogously, if we suppose that there exists a $2$-dimensional invariant subspace $\pi$ containing $e_3$, we get an stratum made of representations conjugated to one of the form
$$
	(A_k) = \left(\left(\begin{array}{c c | c } \Delta_k & 0 & 0  \\ \omega_k & a_k & b_k \\ \hline \eta_k & c_k & d_k \end{array}\right)\right),
$$ 
If we denote this stratum by $\mathcal{R}_{\SL(\CC^3_{\bullet})}^{2, \beta}(\Free{{d-1}})$, the isomorphism $\mathcal{R}_{\SL(\CC^3_{\bullet})}^{2, \alpha}(\Free{{d-1}}) \ni (A, B) \mapsto (A^t, B^t) \in \mathcal{R}_{\SL(\CC^3_{\bullet})}^{2, \beta}(\Free{{d-1}})$ shows that $e(\mathcal{R}_{\SL(\CC^3_{\bullet})}^{2, \beta}(\Free{{d-1}})) = e(\mathcal{R}_{\SL(\CC^3_{\bullet})}^{2, \alpha}(\Free{{d-1}}))$, as computed above. 

However, these spaces are not disjoint. Their intersection is given by representations $\rho$ that contain an invariant line $\ell_\rho \subset \langle e_1, e_2\rangle$ and an invariant plane $\pi_\rho \subset \CC^3$ containing $e_3$. Hence, we have a map
$$
	\mathcal{R}_{\SL(\CC^3_{\bullet})}^{2, \alpha}(\Free{{d-1}}) \cap \mathcal{R}_{\SL(\CC^3_{\bullet})}^{2, \beta}(\Free{{d-1}}) \to \mathbb{P}^1 \times \mathbb{P}^1, \quad \rho \mapsto (\ell_\rho, \pi_\rho^\perp),
$$
where we have identified the planes $\pi$ of $\CC^3$ containing $e_3$ with $\mathbb{P}^1$ through its `normal' direction $\pi^\perp \in \langle e_1, e_2 \rangle$.

This map is surjective, but its fiber depends on the intersection of $\ell_\rho$ and $\pi_\rho$. On the locus $B_1 = \{(\ell, \pi^\perp) \in \mathbb{P}^1 \times \mathbb{P}^1 \,\mid\,  \ell \cap \pi = \{0\}\}$, the fibration is locally trivial since we have a splitting into a direct sum of a $1$-dimensional and a $2$-dimensional invariant subspaces. In particular, for $\ell = \langle e_1 \rangle$ and $\pi = \langle e_2, e_3\rangle$, we get representations of the form
\begin{equation}\label{eq:semisimple}
	(A_k) = \left(\left(\begin{array}{c c | c } \Delta_k & 0 & 0  \\ 0 & a_k & b_k \\ \hline 0 & c_k & d_k \end{array}\right)\right),
\end{equation}
with $(b_1, \ldots, b_{d-1}), (c_1, \ldots, c_{d-1}) \neq (0, \ldots, 0)$.
Hence, the $E$-polynomial of the fiber is
\begin{align*}
e(\mathcal{R}_{\GL_2}(\Free{{d-1}}) \cap \{(b_k), (c_k) \neq 0\}) &  = e(\mathcal{R}_{\GL_2}(\Free{{d-1}})) - e(\mathcal{R}_{\GL_2}(\Free{{d-1}}) \cap \{(b_k) = 0\}) \\ & \quad  - e(\mathcal{R}_{\GL_2}(\Free{{d-1}}) \cap \{(c_k) = 0\}) + e(\mathcal{R}_{\GL_2}(\Free{{d-1}}) \cap \{(b_k) = (c_k) = 0\}) \\
& = (q-1)^{2d-2}(q^{d-1}(q+1)^{d-1} - 2q^{d-1} + 1).
\end{align*}

Analogously, on the locus $B_2 = \{(\ell, \pi^\perp) \in \mathbb{P}^1 \times \mathbb{P}^1 \,\mid\,  \ell \subset \pi\} \cong \mathbb{P}^1$, we have another locally trivial fibration whose fiber over  $\ell = \langle e_1 \rangle$ and $\pi = \langle e_1, e_3\rangle$ is given by representations of the form
\begin{equation}\label{eq:semisimple}
	(A_k) = \left(\left(\begin{array}{c c | c } \Delta_k & \alpha_k & \beta_k  \\ 0 & a_k & 0  \\ \hline 0 & c_k & d_k \end{array}\right)\right),
\end{equation}
with $(b_1, \ldots, b_{d-1}), (c_1, \ldots, c_{d-1}) \neq (0, \ldots, 0)$. There representations are exactly isomorphic to the stratum $(b_k) = 0$ of $\mathcal{R}_{\SL(\CC^3_{\bullet})}^{2, \alpha}(\Free{{d-1}})$, which we know has $E$-polynomial equal to $q^{d-1}(q-1)^{2d-2}(q^{d-1} - 1)^2$.

Therefore, taking into account that $e(B_2) = q+1$ and $e(B_1) = e(\mathbb{P}^1)^2 - e(B_2) = q(q+1)$, we get that
\begin{align*}
	e(\mathcal{R}_{\SL(\CC^3_{\bullet})}^{2, \alpha}(\Free{{d-1}}) \cap \mathcal{R}_{\SL(\CC^3_{\bullet})}^{2, \beta}(\Free{{d-1}})) & =q(q+1)(q-1)^{2d-2}(q^{d-1}(q+1)^{d-1} - 2q^{d-1} + 1) \\
	& \qquad + (q+1)q^{d-1}(q-1)^{2d-2}(q^{d-1} - 1)^2.
\end{align*}

Therefore, by the inclusion-exclusion principle, we finally get that
\begin{align*}
	e(\mathcal{R}_{\SL(\CC^3_{\bullet})}^{2}(\Free{{d-1}})) & = e(\mathcal{R}_{\SL(\CC^3_{\bullet})}^{2, \alpha}(\Free{{d-1}})) + e(\mathcal{R}_{\SL(\CC^3_{\bullet})}^{2, \beta}(\Free{{d-1}})) - e(\mathcal{R}_{\SL(\CC^3_{\bullet})}^{2, \alpha}(\Free{{d-1}}) \cap \mathcal{R}_{\SL(\CC^3_{\bullet})}^{2, \beta}(\Free{{d-1}})) \\
	&= (q-1)^{2d-2}(q+1)  \left(q^{d-1}(q^{d-1} - 1)^2 + (2q^{2d-2}-q) (q^{d-1}(q+1)^{d-1} - 2q^{d-1} + 1)\right).
\end{align*}

	\item Irreducible representations. The stabilizer of these representations is the finite subgroup $\bm{\mu}_{3} \subset H_4$ of multiples of the identity, and thus the action of $H_4 / \bm{\mu}_{3}  = \GL_2( \CC) / \bm{\mu}_{3} $ is free. The $E$-polynomial of this stratum is given by
\begin{align*}
	e(\mathcal{R}_{\SL(\CC^3_{\bullet})}^{\textrm{irr}}(\Free{{d-1}})) &  = e(\mathcal{R}_{\SL(\CC^3_{\bullet})}(\Free{{d-1}})) - e(\mathcal{R}_{\SL(\CC^3_{\bullet})}^{1}(\Free{{d-1}}))  - e(\mathcal{R}_{\SL(\CC^3_{\bullet})}^{2}(\Free{{d-1}})) \\ 
& = (q-1)^{2d-2} \left( (q^2+q+1)^{d-1}(q+1)^{d-1}q^{3d-3}-(q^2+q)^{d-1}(2q^{2d-2}-1) \right. \\ &  \quad  \left. - (q+1)\left( q^{d-1}(q^{d-1} - 1)^2 + (2q^{2d-2}-q) ((q^2+q)^{d-1} - 2q^{d-1} + 1) \right) \right).
\end{align*}
Hence, taking into account that $e(\GL_2( \CC) / \bm{\mu}_{3}) =  e(\GL_2( \CC))$, the $E$-polynomial of the quotient is given by
\begin{align*}
	e(\mathcal{R}_{\SL(\CC^3_{\bullet})}^{\textrm{irr}}(\Free{{d-1}}) \sslash H_4)  & = e(\mathcal{R}_{\SL(\CC^3_{\bullet})}^{\textrm{irr}}(\Free{{d-1}}))/e(\GL_2(\CC)) \\ & = (q-1)^{2d-4} \left( (q^2+q+1)^{d-1}(q+1)^{d-2}q^{3d-4}-(q^2+q)^{d-2}(2q^{2d-2}-1) \right. \\ & \quad  \left. - q^{d-2}(q^{d-1} - 1)^2 - (2q^{2d-3}-1) ((q^2+q)^{d-1} - 2q^{d-1} + 1)  \right).
\end{align*}

\end{itemize}

Putting all together, we get
\begin{align*}
	e(\mathcal{R}_{\SL(\CC^3_{\bullet})}(\Free{{d-1}}) \sslash H_4)  & = e(\mathcal{R}_{\SL(\CC^3_{\bullet})}^{1}(\Free{{d-1}}) \sslash H_4)  + e(\mathcal{R}_{\SL(\CC^3_{\bullet})}^{2}(\Free{{d-1}}) \sslash H_4)  + e(\mathcal{R}_{\SL(\CC^3_{\bullet})}^{\textrm{irr}}(\Free{{d-1}}) \sslash H_4) \\ 
	& = \frac{1}{2}\left( (q-1)^{2d-2} + (q^2-1)^{d-1}\right) + (q-1)^{2d-3}((q^2+q)^{d-1} - 2q^{d-1} +1) \\
	& \quad + (q - 1)^{d-1}\big((q - 1)^{d-2}q^{d-2}((q + 1)^{d-2} - 1) + \frac12 (q - 1)^{d-2} - \frac12 (q + 1)^{d-2} \big)   \\ & \quad + (q-1)^{2d-4} \left( (q^2+q+1)^{d-1}(q+1)^{d-2}q^{3d-4}-(q^2+q)^{d-2}(2q^{2d-2}-1) \right. \\ & \quad   \left. - q^{d-2}(q^{d-1} - 1)^2 - (2q^{2d-3}-1) ((q^2+q)^{d-1} - 2q^{d-1} + 1)  \right),
\end{align*}
which gives the desired result.

\end{proof}

\begin{prop}\label{prop:h4-h4}
We have
$$
e(H_4^{d-1}\sslash H_4) = (q-1)^{d-1}\left( (q^3-q)^{d-2}-(q^2-q)^{d-2}+q\left( \frac{1}{2}(q+1)^{d-2}+\frac{1}{2}(q-1)^{d-2} \right) \right) ,
$$
where the action is given by simultaneous conjugation, $(A_1,\ldots,A_{d-1})\rightarrow (PA_1P^{-1},\ldots,PA_{d-1}P^{-1})$, $P\in H_4$ and $H_4$ 
is defined in \eqref{def:subgroupsSL3}.
\end{prop}

\begin{proof}
Given $Q=\left(\begin{array}{c | c} P & 0 \\ \hline 0 & \det(P)^{-1} \end{array}\right)$, and $(A_k) \in H_4^{d-1}$, the action can be described as
$$
(A_k)=\left(\left( \begin{array}{c | c}  B_k & \begin{matrix} 0 \\ 0 \end{matrix} \\ \hline  \begin{matrix} 0 & 0 \end{matrix} & \det(B_k)^{-1} \end{array} \right) \right) \longrightarrow Q(A_k)Q^{-1} \left( \begin{array}{c | c}  PB_kP^{-1} & \begin{matrix} 0 \\ 0 \end{matrix} \\ \hline  \begin{matrix} 0 & 0 \end{matrix} & \det(B_k)^{-1} \end{array} \right).
$$
The quotient is thus isomorphic to the free $\GL_2(\CC)$-character variety of 
$d-1$ elements, whose $E$-polynomial was computed in \cite{mozgovoyreineke:2015}.
\end{proof}

\section{$E$-polynomials of $\X^d_{\SL_3}(n,m)$}\label{sec:E-poly-rank3}

We compute the $E$-polynomial of each stratum $\cW_{H_1,H_2}$ following the description given in Section \ref{sec:sl3charactervariety} and the different choices for $H_1,H_2$ following the definitions given in \eqref{def:subgroupsSL3}.

\begin{itemize} 
\item[$\bm{H_1},\bm{{\SL}_3}$.] The $E$-polynomial of the stratum coincides with the one of the irreducible locus $\X(n,m)^{\ast}$.
From \cite{munozporti:2016}, the $E$-polynomial of each of the irreducible components of dimension $4$ is $q^4+4q^3-9q^2-3q+12$, and there are $\frac{1}{12}(m-1)(m-2)(n-1)(n-2)$ of them. Non-maximal components are isomorphic to $(\CC^{\ast})^2$ minus a line with two points removed, so their $E$-polynomial is equal to $(q-1)^2-(q-2)=q^2-3q+3$, and there are $\frac{1}{2}(n-1)(m-1)(n+m-4)$ of them. We thus get that 
\begin{align*}
e(\cW_{H_1,\SL_3}) & = (q^3-q)^{d-1}(q^5-q^3)^{d-1}   \left(\frac{1}{12}(m-1)(m-2)(n-1)(n-2) (q^4+4q^3-9q^2-3q+12)  \right.\\ 	 
& \quad \left. +\frac{1}{2}(n-1)(m-1)(n+m-4)(q^2-3q+3) \right).
\end{align*}

\item[$\bm{H_2},\bm{{\SL}_3}$.] The $E$-polynomial of the strata which come from partially reducible representation is computed as follows. From the description given in \cite{munozporti:2016},
$$
e(\cW_{H_2,\SL_3}) = e\left(\left( \X(n,m)_{\SL_2}^{\ast}\times \bm{\mu}_{3mn} \right) / \bm{\mu}_2\right) \, e(\sltres^{d-1}\sslash H_2).
$$
We observe that the $\bm{\mu}_2$-action interchanges components in $\X(n,m)_{\SL_2}^*$ leaving $\bm{\mu}_{3mn}$ invariant. When $n,m$ are odd, using Proposition \ref{eqn:epolySL3H2} and the description of $\X_{\SL_2}^*$ given in Section \ref{sec:sl2geodesc}, we get
\begin{align*}
e(\cW_{H_2,\SL_3}) & = e(\X(n,m)_{\SL_2}^{\ast}/\bm{\mu}_2)e(\bm{\mu}_{3mn})e(\sltres^{d-1}\sslash H_2), \\
& = \frac{3mn}{4}(m-1)(n-1)(q-2)e(\sltres^{d-1}\sslash H_2).
\end{align*}
When $n$ is even, there are $(m-1)/2$ extra components that are preserved by $\bm{\mu}_2 \subset \bm{\mu}_{3mn}$. These components are parametrized by $r \in \CC-\lbrace 0,1 \rbrace$, and the $\bm{\mu}_2$-action takes $r\to 1-r$, so the quotient is parametrized by $\CC^*$. Thus, when $n$ is even
\begin{align*}
e(\cW_{H_2,\SL_3}) & = e(\X(n,m)^{\ast}\times \bm{\mu}_{3mn}/\bm{\mu}_2)e(\sltres^{d-1}\sslash H_2) \\
& = 3mn\left(\frac{1}{4}(m-1)(n-2)(q-2))+\frac{(m-1)}{2}(q-1) \right)e(\sltres^{d-1}\sslash H_2).
\end{align*}
Both cases can be written jointly as
\begin{align*}
e(\cW_{H_2,\SL_3}) & = 3mn\left(\left\lfloor\frac{m-1}{2}\right\rfloor \left\lfloor \frac{n-1}{2}\right\rfloor(q-2)+\delta_n(m-1)(q-1)\right)e(\SL_3(\CC)^{d-1}\sslash H_2),
\end{align*}
where $\delta_n = 1$ if $n$ is even and $\delta_n = 0$ if $n$ is odd.

\item[$\bm{H_2},\bm{H_4}.$] The description given in \eqref{eqn:VH2H4} provides a similar computation to the one that was performed in the previous case, since the $\bm{\mu}_2$-action is the same and leaves $\CC^{\ast}-\bm{\mu}_{3mn}$ invariant. We thus obtain
\begin{equation*}
e(\cW_{H_2,H_4}) = (q-3mn-1)\left(\left\lfloor\frac{m-1}{2}\right\rfloor\left\lfloor\frac{n-1}{2}\right\rfloor(q-2)+\delta_n(m-1)(q-1)\right)e(H_4^{d-1}\sslash H_2),
\end{equation*}
Now, observe that the action of $H_2 \cong \CC^*$ on $H_4^{d-1} \cong \GL_2( \CC)^{d-1}$ is trivial, so we get $e(H_4^{d-1}\sslash H_2) = e(\GL_2( \CC))^{d-1} = (q-1)^{2d-2}(q^2+q)^{d-1}$.

\item[$\bm{{\SL}_3},\bm{{\SL}_3}$.] Since $\cW_{\SL_3,\SL_3}$ consists of three copies of the character variety of the free group of $d-1$ generators, we get its $E$-polynomial from \cite{lawtonmunoz:2016}. Hence,
\begin{align*}
e(\cW_{\SL_3,\SL_3}) & = 3\left( \vphantom{\frac{1}{3}}(q^8-q^6-q^5+q^3)^{d-2} + (q-1)^{2d-4}(q^{3d-6}-q^{d-1}) \right. \\
& \quad +\frac{1}{6}(q-1)^{2d-4}q(q+1)+\frac{1}{2}(q^2-1)^{d-2}q(q-1) \\ & \quad \left. 
+\frac{1}{3}(q^2+q+1)^{d-2}q(q+1)-(q-1)^{d-2}q^{d-2}(q^2-1)^{d-2}(2q^{2d-4}-q) \right).
\end{align*}

\item[$\bm{H_4},\bm{{\SL}_3}$.] $V^0_{H_4,\SL_3}$ is a finite collection of points, so
$$
e(\cW_{H_4,\SL_3}) = (3mn-3)e(\SL_3(\CC)^{d-1}\sslash H_4).
$$ 
The $E$-polynomial $e(\SL_3(\CC)^{d-1}\sslash H_4)$ is given by Proposition \ref{prop:SL3-h4}.

\item[$\bm{H_4},\bm{H_4}$.] In this case, 
\begin{align*}
e(\cW_{H_4,H_4})& =e(\CC^{\ast}-\bm{\mu}_{3mn})e(H_4^{d-1}\sslash H_4) = (q-3mn-1)e(H_4^{d-1}\sslash H_4).
\end{align*}
The $E$-polynomial $e(H_4^{d-1}\sslash H_4)$ is given by Proposition \ref{prop:h4-h4}.

\item[$\bm{H_3},\bm{{\SL}_3}$.] $V^0_{H_3,\SL_3}$ is a collection of $3m^2n^2-9mn+6$ points, and $N_{H_3}= S_3$ acts freely on it. Hence, we have
$$
	e(\cW_{H_3,{\SL}_3}) = e(V^0_{H_3,\SL_3}/S_3) e(\SL_3^{d-1} \sslash H_3) = \frac{3m^2n^2-9mn+6}{6}e(\SL_3^{d-1} \sslash H_3).
$$
The $E$-polynomial of $e(\SL_3^{d-1} \sslash H_3)$ is $ e(\mathcal{X}_{(1,1,1)}(\Free{d-1}))/(q-1)^{d-1}$, where the latter is computed in Example \ref{ex:X111}.

\item[$\bm{H_3},\bm{H_4}$.] In this case, from \cite{gonmun:2022},
$$
e(V^0_{H_3,H_4}) = \left( \left\lfloor \frac{mn}{2} \right\rfloor (q-1)- \frac{3mn(mn-1)}{2} \right) T + \left( \left\lfloor \frac{mn-1}{2} \right\rfloor (q-1)- \frac{3mn(mn-1)}{2} \right) N,
$$
 and from the fact that $H_4^{d-1}\sslash H_3 = \GL_{2}^{d-1} \sslash (\CC^*)^2$, its equivariant $E$-polynomial is $e_{\ZZ_2}(H_4^{d-1}\sslash H_3) = e_{\ZZ_2}(\mathcal{X}_{(1,1)}(\Free{{d-1}}))$, where this polynomial is computed in Example \ref{ex:X11-Z2}. Thus, we have
\begin{align*}
 e_{\ZZ_2}(H_4^{d-1}\sslash H_3) & = \left({q^{d}\left(q + 1\right)}^{d-2} {\left(q - 1\right)}^{2d-3} - q^{d-1} {\left(q - 1\right)}^{2  d-3} + \frac{1}{2} \left(q{\left(q - 1\right)}^{2  d-3}  + {\left(q^{2} - q\right)} {\left(q^{2} - 1\right)}^{d-2}\right)\right)\,T \\
	& \quad + \left( q^{d-1}{\left(q + 1\right)}^{d-2} {\left(q - 1\right)}^{2  d-3} -  q^{d-1}{{\left(q - 1\right)}^{2  d-3}} + \frac{1}{2} \left({q\left(q - 1\right)}^{2d-3} - {\left(q^{2} - q\right)} {\left(q^{2} - 1\right)}^{d-2}\right)\right)\,N.
\end{align*}

Now, recall that $\left\lfloor \frac{k}{2} \right\rfloor + \left\lfloor \frac{k-1}{2} \right\rfloor = k -1$ for any positive integer $k$ and that $\left\lfloor \frac{k}{2} \right\rfloor - \left\lfloor \frac{k-1}{2} \right\rfloor = \delta_k$, where again $\delta_k = 1$ if $k$ is even and $\delta_k = 0$ otherwise. We get that $e(\cW_{H_3,H_4})$ is the $T$-coefficient of $e_{S_3}(V^0_{H_3,H_4})\otimes e_{S_3}(H_4^{d-1}\sslash H_3)$ and thus
\begin{align*}
e(\cW_{H_3,H_4}) & = (mn-1)\left( q-3mn-1 \right) \left(\frac{1}{2} q{\left(q - 1\right)}^{2d-3} - q^{d-1} {\left(q - 1\right)}^{2d-3}\right) + \frac{\delta_{nm}}{2} q(q-1)^{d} (q + 1)^{d-2}  \\
& \quad + \left( q\left\lfloor \frac{mn}{2} \right\rfloor + \left\lfloor \frac{mn-1}{2} \right\rfloor\right) {q^{d-1}\left(q + 1\right)}^{d-2} {\left(q - 1\right)}^{2d-2} 
 -\frac{3mn(mn-1)}{2} {q^{d-1}\left(q + 1\right)}^{d-1} {\left(q - 1\right)}^{2d-3}.
\end{align*}
\item[$\bm{H_3},\bm{H_3}$.] Since
$$
\cW_{H_3,H_3}=(V^0_{H_3,H_3}\times H_3^{d-1})/ S_3\, ,
$$
we may use the equivariant polynomials of $V^0_{H_3,H_3}$ from \cite{gonmun:2022}, so
\begin{align*}
e_{S_3}(V^0_{H_3,H_3}) & = \left( q^2-q-\left\lfloor \frac{mn}{2} \right\rfloor (q-1)+m^2n^2 \right) T - \left( \left\lfloor \frac{mn-1}{2} \right\rfloor (q-1)-m^2n^2 \right) S  \\ & \quad -((mn+1)(q-1)-2m^2n^2) D,
\end{align*}
and also $ e_{S_3}(H_3^{d-1})=(e_{S_3}(H_3))^{d-1}= (q^2 T+S-qD)^{d-1}$, so
\begin{align}
e_{S_3}(\cW_{H_3,H_3}) & = \left( \left( q^2-q-\left\lfloor \frac{mn}{2} \right\rfloor (q-1)+m^2n^2 \right) T - \left( \left\lfloor \frac{mn-1}{2} \right\rfloor (q-1)-m^2n^2 \right) S \right. \nonumber \\ & \quad \left. -((mn+1)(q-1)-2m^2n^2) D \vphantom{\frac{mn-1}{2}}\right)\otimes (q^2T +S-qD)^{d-1}. \label{eqn:equivariantWH3H3}
\end{align}
From \cite{lawtonmunoz:2016}, we get that
\begin{align*}
(q^2T+S-qD)^{d-1} & = \left( \frac{(q^2-1)^{d-1}}{2}+\frac{1}{6}(q-1)^{2d-2}+\frac{1}{3}(q^2+q+1)^{d-1} \right) T \\ & \quad  + \left( -\frac{(q^2-1)^{d-1}}{2}+\frac{1}{6}(q-1)^{2d-2}+\frac{1}{3}(q^2+q+1)^{d-1} \right) S \\ & \quad +\frac{1}{3} \left((q-1)^{2d-2}-(q^2+q+1)^{d-1}\right) D,
\end{align*}
using the multiplication table in Example \ref{ex:6}.

Subtituting its expression into \eqref{eqn:equivariantWH3H3}, its $T$-coefficient provides the $E$-polynomial of $\cW_{H_3,H_3}$, which equals
\begin{align*}
e(\cW_{H_3,H_3}) & = \left( q^2-q-\left\lfloor \frac{mn}{2} \right\rfloor (q-1)+m^2n^2 \right) \left( \frac{(q^2-1)^{d-1}}{2}+\frac{1}{6}(q-1)^{2d-2}+\frac{1}{3}(q^2+q+1)^{d-1} \right) \\ & \quad - \left( \left\lfloor \frac{mn-1}{2} \right\rfloor (q-1)-m^2n^2 \right)\left( -\frac{(q^2-1)^{d-1}}{2}+\frac{1}{6}(q-1)^{2d-2}+\frac{1}{3}(q^2+q+1)^{d-1} \right) \\ &  \quad -\frac{1}{3}((mn+1)(q-1)-2m^2n^2) \left((q-1)^{2d-2}-(q^2+q+1)^{d-1}\right)  \\ 
& = \frac{(q^2-1)^{d-1}}{2}(q-1)\left(q-\delta_{mn} \right) + \frac{1}{6}(q-1)^{2d-2} \left((q-1)\left(q-3mn-1\right) +6m^2n^2 \right) \\ & \quad 
+\frac{1}{3}(q^2+q+1)^{d-1} (q-1)(q+2).
\end{align*}
\end{itemize}

Adding up all the contributions, we finally get
\begingroup
\allowdisplaybreaks
\begin{align*}
e(\X_{\SL_3(\CC)}(K_{n,m}^d)) = & \phantom{+}  \frac{(q^3-q)^{d-1}}{12}(q^5-q^3)^{d-1} \left((m-1)(m-2)(n-1)(n-2) (q^4+4q^3-9q^2-3q+12) \right.  \\  & \left.  + \,6(n-1)(m-1)(n+m-4)(q^2-3q+3) \right) \\
& + 3mn\left(\left\lfloor\frac{m-1}{2}\right\rfloor \left\lfloor \frac{n-1}{2}\right\rfloor(q-2)+\delta_n(m-1)(q-1)\right) (q^3-q)^{d-1} \\ & \quad \left(q^{3d-3}(q+1)^{d-1}(q-1)^{d-2} -(2q^{2d-2}-1)(q-1)^{d-2}+(q-1)^{d-1} \right) \\ 
& + (q-3mn-1)\left(\left\lfloor\frac{m-1}{2}\right\rfloor\left\lfloor\frac{n-1}{2}\right\rfloor(q-2)+\delta_n(m-1)(q-1)\right) (q-1)^{2d-2}(q^2+q)^{d-1} \\ 
& +3(q^8-q^6-q^5+q^3)^{d-2} + 3(q-1)^{2d-4}(q^{3d-6}-q^{d-1}) \\ & +\frac{1}{2}(q-1)^{2d-4}q(q+1)+\frac{3}{2}(q^2-1)^{d-2}q(q-1) \\ &  
+(q^2+q+1)^{d-2}q(q+1)-3(q-1)^{d-2}q^{d-2}(q^2-1)^{d-2}(2q^{2d-4}-q)  \\ 
& + (3mn-3)\left( \vphantom{\frac{1}{3}} \frac{1}{2}\left( (q-1)^{2d-2} + (q^2-1)^{d-1}\right) + (q-1)^{2d-3}((q^2+q)^{d-1} - 2q^{d-1} +1) \right. \\& + (q - 1)^{d-1}\big((q - 1)^{d-2}q^{d-2}((q + 1)^{d-2} - 1) + \frac12 (q - 1)^{d-2} - \frac12 (q + 1)^{d-2} \big)   \\ &   + (q-1)^{2d-4} \left( (q^2+q+1)^{d-1}(q+1)^{d-2}q^{3d-4}-(q^2+q)^{d-2}(2q^{2d-2}-1) \right. \\ &  \left. -\, q^{d-2}(q^{d-1} - 1)^2 - (2q^{2d-3}-1) ((q^2+q)^{d-1} - 2q^{d-1} + 1))  \vphantom{\frac{1}{3}}\right) \\
& +(q-3mn-1)(q-1)^{d-1}\left( (q^3-q)^{d-2}-(q^2-q)^{d-2}+ \frac{1}{2}q(q+1)^{d-2}+\frac{1}{2}q(q-1)^{d-2} \right) \\ & + \frac{m^2n^2-3mn+2}{2} \left( \vphantom{\frac{1}{3}} (q-1)^{2d-2}+3(q-1)^{2d-3}((q^2+q)^{d-1}-2q^{d-1}+1) \right. \\ 
	& \left. +\,(q-1)^{2d-4}\left((q^2+q+1)^{d-1}(q+1)^{d-1}q^{3d-3}-3(q+1)^{d-1}(2q^{3d-3}-q^{d-1}) \right. \right. \\ & \left. \left. +6q^{d-1}(q^{2d-2}-1)+2  \right) \vphantom{\frac{1}{3}} \right) + (mn-1)\left( q- 3mn-1 \right) \left(\frac{1}{2} q{\left(q - 1\right)}^{2d-3} - q^{d-1} {\left(q - 1\right)}^{2d-3}\right) \\ & + \frac{\delta_{nm}}{2} q(q-1)^{d} (q + 1)^{d-2}  + \left( q\left\lfloor \frac{mn}{2} \right\rfloor + \left\lfloor \frac{mn-1}{2} \right\rfloor\right) {q^{d-1}\left(q + 1\right)}^{d-2} {\left(q - 1\right)}^{2d-2} 
 \\  & -\frac{3mn(mn-1)}{2} {q^{d-1}\left(q + 1\right)}^{d-1} {\left(q - 1\right)}^{2d-3} + \frac{(q^2-1)^{d-1}}{2}(q-1)\left(q-\delta_{mn} \right)  \\ & + \frac{1}{6}(q-1)^{2d-2} \left((q-1)\left(q-3mn-1\right) +6m^2n^2 \right) +\frac{1}{3}(q^2+q+1)^{d-1} (q-1)(q+2),
\end{align*}
\endgroup
for all $d, n, m \geq 1$ with $n$ and $m$ coprime and $m$ odd.

\begin{rem}
The previous calculation generalizes the two known cases of $E$-polynomials of torus links.
\begin{itemize}
	\item For $d = 1$, we recover the case of torus knots, matching \cite[Theorem 8.3]{munozporti:2016}.
	\item For $d=2$ and $n=1$, we recover the so-called twisted Hopf links, as studied in \cite{gonmun:2022}. The previous formula gives
\begin{align*}
e(\X^2_{\SL_3}(m,1)) &= q^4 + q^2 + 1 + \frac{1}{2}(n^2 - 3n + 2)\left(q^6 + 2q^5 - 4q^4 + q^3 + 3q^2 - 3q + 2\right) \\
& \qquad + 3(n - 1)(q^4 - q^3 + q^2 - q + 1) + (n - 1)(q-1)\left(q^3 - 2q^2 + q \right) -  \left\lfloor \frac{n-1}{2} \right\rfloor (q^3 - 2q^2 + 1)(q-1).
\end{align*}
This formula corrects the one computed in the original version \cite{gonmun:2022}, for which the stratum $\mathfrak{M}_{2,1}^{\textrm{red}}(H_n, \SL_3(\CC))$ was incorrectly computed. The current version of that manuscript in arXiv corrects this error.
\end{itemize}
\end{rem}

\end{document}